%% file: paper.tex
\documentclass[11pt]{article}
\usepackage{amsfonts}
\usepackage{amsthm}
\input{epsf.sty}
\usepackage{latexsym,amsmath,amsfonts,amscd}
\usepackage{epsfig}
\usepackage{changebar}
\usepackage{pstricks}
\usepackage{pst-plot}
\usepackage{multirow}
\usepackage{subfigure}
\usepackage{placeins}
\usepackage{amssymb}
\usepackage{mathrsfs}
\usepackage[title]{appendix}
\usepackage{bm}
\usepackage{hyperref}
\usepackage{geometry}
\usepackage{tikz} %
\usepackage{listings}
\usepackage{color}

\definecolor{dkgreen}{rgb}{0,0.6,0}
\definecolor{gray}{rgb}{0.5,0.5,0.5}
\definecolor{mauve}{rgb}{0.58,0,0.82}

\usepackage{multirow,bigstrut}
\usepackage{enumitem}
\usepackage{booktabs}
\usepackage{listings}

\numberwithin{equation}{section}

\theoremstyle{plain}
\newtheorem{thm}{Theorem}[section]

\newtheorem{prop}[thm]{Proposition}

\theoremstyle{definition}

\newcommand{\brac}[1]{\left(#1\right)}

\def\half{\frac 1 2}

\newcommand{\bk}{\bm{k}}
\newcommand{\bx}{\bm{x}}

\newfont{\iams}{msbm9}

\newcommand{\commentbis}[1]{}
\newcommand{\be}{\begin{eqnarray}}
\newcommand{\ee}{\end{eqnarray}}
\newcommand{\beno}{\begin{eqnarray*}}
\newcommand{\eeno}{\end{eqnarray*}}

\newcommand{\barr}[1]{\begin{array}{#1}}
\newcommand{\earr}{\end{array}}

\newcommand{\beq}{\begin{equation}}
\newcommand{\eeq}{\end{equation}}
\newcommand{\beqa}{\begin{eqnarray}}
\newcommand{\eeqa}{\end{eqnarray}}

\newcommand{\bv}{{\bf v}}

\newcommand{\bV}{{\bf V}}

\newcommand{\bzero}{{\bm{0}}}
\newcommand{\bone}{{\bm{1}}}
\newcommand{\bl}{{\bm{l}}}
\newcommand{\bi}{{\bm{i}}}

\newcommand{\bj}{{\bm{j}}}
\newcommand{\bW}{{\bm{W}}}

\newcommand{\ba}{{\bm{a}}}
\newcommand{\bal}{{\bm{\alpha}}}
\newcommand{\bb}{{\bm{\beta}}}

\allowdisplaybreaks[4]

\title{Adaptive sparse grid discontinuous Galerkin method: review and software implementation}
\author{
Juntao Huang
\thanks{Department of Mathematics and Statistics, Texas Tech University, Lubbock, TX 79409 U.S.A.
{\tt juntao.huang@ttu.edu}.}
\and
Wei Guo
\thanks{Department of Mathematics and Statistics, Texas Tech University, Lubbock, TX, 70409 U.S.A. 
{\tt weimath.guo@ttu.edu}.
Research is supported by NSF grant DMS-2111383, Air Force Office of Scientific Research FA9550-18-1-0257.}
\and
Yingda Cheng
\thanks{Department of Mathematics, Department of  Computational Mathematics, Science and Engineering, Michigan State University,
East Lansing, MI 48824 U.S.A.
 {\tt ycheng@msu.edu}. Research is supported by NSF grant DMS-2011838.}
}

\begin{document}
\baselineskip=1.6pc
\maketitle

\begin{abstract}
This paper reviews the adaptive sparse grid discontinuous Galerkin (aSG-DG) method for computing high dimensional partial differential equations (PDEs) and its software implementation. The C\texttt{++} software package called AdaM-DG, implementing the aSG-DG method, is available on Github at \url{https://github.com/JuntaoHuang/adaptive-multiresolution-DG}. The package is capable of treating a large class of high dimensional linear and nonlinear PDEs. We review the essential components    of the algorithm and the functionality of the software, including the multiwavelets used, assembling of bilinear operators, fast matrix-vector product for data with hierarchical structures. We further demonstrate the performance of the package by reporting numerical error and CPU cost for several benchmark test, including linear transport equations, wave equations and Hamilton-Jacobi equations.
\end{abstract}

\bigskip

\bigskip

{\bf Key Words: } Adaptive sparse grid, discontinuous Galerkin, high dimensional partial differential equation, software development

\pagenumbering{arabic}
\newpage

\input{intro}

\input{basis}

\input{operator}

\input{numerical}

\input{conclusion}

\newpage
\bibliographystyle{abbrv}
\bibliography{ref,ref_cheng,ref_cheng_2}

\end{document}

%% file: intro.tex
\section{Introduction}\label{sec:intro}
\setcounter{equation}{0}
\setcounter{figure}{0}
\setcounter{table}{0}

In recent years, we initiated a line of research to develop adaptive sparse grid discontinuous Galerkin (aSG-DG) method for computing high dimensional partial differential equations (PDEs). This paper serves as a review of the fundamental philosophy behind the algorithm, and more importantly its numerical implementation.  

It is well known that any grid based solver for  high dimensional PDEs suffers from the  \emph{curse of dimensionality} \cite{bellman1961adaptive}. This term refers to the fact that the computational degree of freedom (DOF)   scale as  $O(h^{-d})$ for a $d$-dimensional problem, where   $h$ denotes the mesh size in one coordinate direction, and  for a grid based method (e.g. finite difference or finite element method) with a fixed order of accuracy $k,$ this means the dependence of error on the DOF scales as  $O(\textrm{DOF}^{-k/d}).$ As such, when $d \rightarrow \infty,$ the exponent goes to zero regardless of $k,$ and this means the numerical solution will be inaccurate due to the limited computational resources.
To break the curse of dimensionality, there are several possible approaches. One is to use a probabilisitic type method, such as the Monte Carlo algorithms. The drawback of this approach is the loss of accuracy due to the inherent statistical noise. Another approach, which is the one we are taking, is the sparse grid method \cite{bungartz2004sparse, garcke2013sparse},  introduced by Zenger \cite{zenger1991sparse}. The idea relies on a tensor product hierarchical basis representation, which can reduce the degrees of freedom   from   {$O(h^{-d})$ to $O(h^{-1}|\log_{2}h|^{d-1})$ for $d$-dimensional problems without compromising much accuracy.}  This method is very suitable for moderately high dimensional problems, offering a balance between accuracy and computational cost, see  \cite{griebel2005sparse} for a review.  %

Our work is focused on using sparse grid techniques to solve high dimensional PDEs.
Sparse grid   finite element methods \cite{zenger1991sparse, bungartz2004sparse,schwab2008sparse} and spectral methods \cite{griebel2007sparse, gradinaru2007fourier,shen2010sparse,shen2010efficient} are the most well-developed   sparse grid PDE solvers.   Our research, on the other hand,  is inspired by the distinctive advantages of DG method for transport dominated problems, and with the sparse grid technique,   our ultimate goal is the efficient computations of high-dimensional transport dominated problems such as kinetic equations and   Hamilton-Jacobi equations. We start by developing sparse grid DG method for elliptic, linear transport and kinetic problems in \cite{wang2016elliptic, guo2016transport}.
We then developed the adaptive version: the aSG-DG method in \cite{guo2017adaptive}. In \cite{tao2019collocation}, we developed new interpolatory multiwavelets for piecewise polynomial spaces, and used those multiwavelets to compute hyperbolic conservation laws \cite{huang2019adaptive},  wave equations \cite{huang2020adaptive}, and nonlinear dispersive equations \cite{tao2022adaptive, huang2022class}.  Since the underlying mechanism of the aSG method is  multiresolution analysis, we  also call aSG-DG method the adaptive multiresolution DG method, hence the name AdaM-DG (Adaptive Multiresolution DG) as the name of our package.

While developing and analyzing algorithms are important, we feel that software implementation is also crucial in this project. Reducing the DOF is one thing, the success of the reduction in CPU cost and memory is another.
Here, we outline several challenges facing the efficient implementation of the method. First, the aSG-DG schemes rely on  computations using non-local basis functions (the multiwavelets), this means the standard element-wise DG implementation is no longer feasible. One has to formulate and compute the scheme in the global sense. Second, the multiwavelets are hierarchical. This  hierarchical structure  induces ``orthogonality" in some sense, and it has to be exploited for a fast computation. A prominent example is the fast wavelet transform, which incurs linear cost with respect to the DOF. Third, the implementation has to be adaptive. It is well known that the smoothness requirement of the sparse grid is stringent. This means that the  software implementation should be designed with adaptivity in mind. 

To address the aforementioned challenges, we developed the numerical methods and the software AdaM-DG (available at \url{https://github.com/JuntaoHuang/adaptive-multiresolution-DG}) with several key features. Our method uses two sets of multiwavelets: Alpert's multiwavelets \cite{alpert1993}, which are $L^2$ orthonormal and a class of interpolatory wavelets \cite{tao2019collocation}   for variable coefficients and general nonlinear problems. Many key algorithms rely on fast matrix-vector product exploring the mesh level hierarchy. The Hash table is the underlying fundamental data structures serving the purpose of adaptivity. The code is written with a high level abstraction under 
a uniform treatment for different dimension numbers and encompasses a  universal framework for various equations with different weak formulations. The current package has been used to compute nonlinear hyperbolic conservation laws \cite{huang2019adaptive}, wave equations \cite{huang2020adaptive}, nonlinear Schr\"{o}dinger equations \cite{tao2022adaptive}, Hamilton-Jacobi equations \cite{guo2021adaptive} and nonlinear dispersive equations \cite{huang2022class}, although generalizing this computational module to other applications is achievable with a reasonable modification of the code on the high level.

There are several other sparse grid packages available on the market, mostly for high dimensional function interpolation and integration. For computing high dimensional PDEs, there are two other main packages.
SG\texttt{++} \cite{pflueger10spatially}  is a universal toolbox for spatially adaptive sparse grid methods and the sparse grid combination technique. It provides various low-level and high-level sparse grid functionality allowing one to start using sparse grids with minimal initial implementation effort. The functionality include function interpolation, quadrature, numerical solver for PDEs, data mining and machine learning, and uncertainty quantification. In terms of PDE solvers, it supports the elliptic and parabolic equations using a finite element approach. 
ASGarD (Adaptive Sparse Grid Discretization) \cite{ornlgithub} is a package implementing the adaptive sparse grid DG method  with efficient parallel implementations for both CPU (using OpenMP) and GPU (using CUDA). It was applied to Maxwell equation \cite{d2020discontinuous}, linear advection equation, diffusion equation and advection-diffusion equations, e.g. Fokker-Planck equations. This package only considers linear equations and does not support  the interpolatory multiwavelets and the associated fast algorithms. We also mention  \cite{atanasov2017sparse},  where an   open source Julia library implementing the sparse grid DG method was developed and applied to scalar linear wave equations in high dimensions.

The goal of this paper is to illustrate the main components of the algorithm and the software using concrete examples. Throughout the paper, we will introduce and review the concepts used in the method first, then provide  short descriptions of the associated implementations in the AdaM-DG  package. 
With this in mind, the rest of the paper is organized as follows. In Section \ref{sec:basis}, we review the fundamentals of Alpert's and interpolatory
multiwavelets, their   implementation and the adaptive procedure. In Section \ref{sec:operator}, we discuss the operators used to assemble a PDE solver, paying particular attention to the fast algorithm.
Section \ref{sec:example}, we provide details of how to solve three types of PDEs using the package, and provide benchmark results with CPU time. 
Section \ref{sec:conclusion} concludes the paper by discussing the current status of the software with future improvment.

%% file: basis.tex
\section{Multiwavelets and adaptivity}
\label{sec:basis}

The building blocks of the aSG-DG method are the multiwavelet basis functions and the associated adaptive   procedures. In this section, we will review the two types of multiwavelet basis functions used and their implementations with data structure. We will also go over the details of the adaptive refining and coarsening procedures illustrated by code blocks.

\subsection{Multiwavelets in 1D}

We use two types of multiwavelet bases.  
We will start by reviewing the construction of Alpert's multiwavelet basis functions \cite{alpert1993} on the unit interval $I = [0,1]$. We define a set of nested grids, where the $n$-th level grid $\Omega_n$ consists of $2^n$ uniform cells
\begin{equation*}
  I_{n}^j=(2^{-n}j, 2^{-n}(j+1)], \quad j=0, \ldots, 2^n-1
\end{equation*}
for $n \ge 0.$ %
The usual piecewise polynomial space of degree at most $k\ge1$ on the $n$-th level grid $\Omega_n$ for $n\ge 0$ is denoted by
\begin{equation}\label{eq:DG-space-Vn}
V_n^k:=\{v: v \in P^k(I_{n}^j),\, \forall \,j=0, \ldots, 2^n-1\}.
\end{equation}
Then, we have the nested structure
$$V_0^k \subset V_1^k \subset V_2^k \subset V_3^k \subset  \cdots$$
We can now define the multiwavelet subspace $W_n^k$, $n=1, 2, \ldots $ as the orthogonal complement of $V_{n-1}^k$ in $V_{n}^k$ with respect to the $L^2$ inner product on $[0,1]$, i.e.,
\begin{equation*}
V_{n-1}^k \oplus W_n^k=V_{n}^k, \quad W_n^k \perp V_{n-1}^k.
\end{equation*}
For notational convenience, we let
$W_0^k:=V_0^k$, which is the standard polynomial space of degree up to $k$ on $[0,1]$. Therefore, we have $V_n^k=\bigoplus_{0 \leq l \leq n} W_l^k$.

Now we  define a set of orthonormal basis associated with the space $W_l^k$. The case of mesh level $l=0$ is trivial: we use the normalized shifted Legendre polynomials in $[0,1]$ and denote the basis by $v^0_{i,0}(x)$ for $i=0,\ldots,k$.
When $l>0$, the orthonormal bases in $W_l^k$ are presented in \cite{alpert1993} and denoted by 
$$v^j_{i,l}(x),\quad i=0,\ldots,k,\quad j=0,\ldots,2^{l-1}-1,$$
where the index $l$ denotes the mesh level, $j$ denotes the location of the element, and $i$ is the index for polynomial degrees.
Note that Alpert's multiwavelets are orthonomal, i.e.,
$
\int_0^1 v^j_{i,l}(x)v^{j'}_{i',l'}(x)\,dx=\delta_{ii'}\delta_{ii'}\delta_{jj'}.
$

The second class of basis functions are the interpolatory multiwavlets introduced in \cite{tao2019collocation}. 
Denote the set of interpolation points in the interval $I=[0,1]$ at mesh level 0 by $X_0 = \{ x_i \}_{i=0}^P\subset I$. Here, we assume the number of points in $X_0$ is $(P+1)$. Then the interpolation points at mesh level $n\ge1$, $X_n$ can be obtained correspondingly as
\begin{equation*}
    X_n = \{ x_{i,n}^j := 2^{-n}(x_i+j), \quad i=0,\dots,P, \quad j=0,\dots,2^{n}-1 \}.
\end{equation*}
We require the points to be nested, i.e. 
\begin{equation}
\label{nestpts}
    X_0 \subset X_1 \subset X_2 \subset X_3 \subset \cdots.
\end{equation}

Given the interpolation points, we define the basis functions on the $0$-th level grid as  Lagrange $(K=0)$ or Hermite $(K\ge1)$ interpolation polynomials of degree $\le M:=(P+1)(K+1)-1$ which satisfy the property:
\begin{equation*}
    \phi_{i,l}^{(l')}(x_{i'}) = \delta_{ii'}\delta_{ll'},
\end{equation*}
for $ i,i'=0,\dots,P$ and $l,l'=0,\dots,K$. Here and afterwards, the superscript $(l')$ denotes the $l'$-th order derivative. It is easy to see that
$
\textrm{span} \{ \phi_{i,l}, \quad i=0,\dots,P, \quad l=0,\dots,K \}=V_0^M.  
$
With the basis function at mesh level 0, we can define basis function at mesh level $n\ge1$:
\begin{equation*}
  \phi_{i,l,n}^j(x) :=2^{-nl}\phi_{i,l}(2^nx-j), \quad i=0,\dots,P, \quad l=0,\dots,K, \quad j=0,\dots,2^n-1 
\end{equation*}
which is a complete basis set for $V_n^M.$

Next, we introduce the hierarchical representations. Define $\tilde{X}_0 := X_0$ and {$\tilde{X}_n := X_n\backslash X_{n-1}$} for $n\ge1$, then we have the decomposition
\begin{equation*}
    X_n = \tilde{X}_0 \cup \tilde{X}_1 \cup \cdots \cup \tilde{X}_n .
\end{equation*}
Denote the points in $\tilde{X}_1$ by $\tilde{X}_1=\{ \tilde{x}_i \}_{i=0}^P$. Then the points in $\tilde{X}_n$ for $n\ge1$ can be represented by
\begin{equation*}
    \tilde{X}_n = \{ \tilde{x}_{i,n}^j:=2^{-(n-1)}(\tilde{x}_i+j), \quad  i=0,\dots,P, \quad j=0,\dots,2^{n-1}-1 \}.
\end{equation*}

For notational convenience, we let $\tilde{W}_0^M:=V_0^M.$ The increment function space $\tilde{W}_n^M$ for $n\ge1$ is introduced as a function space    that satisfies
\begin{equation}\label{eq:func-space-sum}
    V_n^M = V_{n-1}^M \oplus \tilde{W}_n^M,
\end{equation}
and is defined through the multiwavelets
 $\psi_{i,l} \in V_1^M$ that satisfies
\begin{equation*}
    \psi_{i,l}^{(l')}(x_{i'}) = 0, \quad \psi_{i,l}^{(l')}(\tilde{x}_{i'}) = \delta_{i,i'}\delta_{l,l'},
\end{equation*}
for $i,i'=0,\dots,P$ and $l,l'=0,\dots,K$. Then  $\tilde{W}_n^M$ is given by
\begin{equation*}
    \tilde{W}_n^M = \textrm{span} \{ \psi_{i,l,n}^j := 2^{-(n-1)l}\psi_{i,l}(2^{n-1}x-j), \quad i = 0,\dots,P, \, l=0,\dots,K, \, j=0,\dots,2^{n-1}-1 \}
\end{equation*}
The explicit expression of the interpolatory multiwavelets basis functions can be found in \cite{tao2019collocation,huang2020adaptive}. 
The algorithm converting between the point values and the derivatives at the interpolation points to hierarchical coefficients   is given in \cite{tao2019collocation}, and by a standard  argument in fast wavelet transform, can be performed with linear complexity.

In the software implementation, we have a base class \verb|Basis| which is determined by three indices \verb|level| $n$, \verb|suppt| $j$ and \verb|dgree| $p$. This is the base class denoting the basis functions in 1D. The  three classes \verb|AlptBasis|, \verb|LagrBasis| and \verb|HermBasis| are inherited from \verb|Basis|, which denote the Alpert's, Lagrange interpolotary and Hermite interpolotary multiwavelets, respectively. The values and the derivatives of the basis functions can be computed through the member functions in the class.
All the basis functions in 1D are collected in the template class \verb|template <class T> AllBasis| in a given order. The key member variable in this class is \verb|std::vector<T> allbasis|, which is composed of all the basis functions, with \verb|T| being \verb|AlptBasis|, \verb|LagrBasis| or \verb|HermBasis|. The total number of basis functions (i.e. the size of \verb|allbasis|) is $(k+1)2^n$ with $k$ being the maximum polynomial degree and $n$ being the maximum mesh level in the computation.

\subsection{Multiwavelets in multidimensions}

Multidimensional case when $d>1$ follows from a tensor-product approach. First we recall some basic notations to facilitate the discussion. For a multi-index $\mathbf{\alpha}=(\alpha_1,\cdots,\alpha_d)\in\mathbb{N}_0^d$, where $\mathbb{N}_0$  denotes the set of nonnegative integers, the $l^1$ and $l^\infty$ norms are defined as 
$
|\bal|_1:=\sum_{m=1}^d \alpha_m, \qquad   |\bal|_\infty:=\max_{1\leq m \leq d} \alpha_m.
$
The component-wise arithmetic operations and relational operations are defined as
$$
\bal \cdot \bb :=(\alpha_1 \beta_1, \ldots, \alpha_d \beta_d), \qquad c \cdot \bal:=(c \alpha_1, \ldots, c \alpha_d), \qquad 2^\bal:=(2^{\alpha_1}, \ldots, 2^{\alpha_d}),
$$
$$
\bal \leq \bb \Leftrightarrow \alpha_m \leq \beta_m, \, \forall m,\quad
\bal<\bb \Leftrightarrow \bal \leq \bb \textrm{  and  } \bal \neq \bb.
$$

By making use of the multi-index notation, we denote by $\bl=(l_1,\cdots,l_d)\in\mathbb{N}_0^d$ the mesh level in a multivariate sense. We  define the tensor-product mesh grid $\Omega_\bl=\Omega_{l_1}\otimes\cdots\otimes\Omega_{l_d}$ and the corresponding mesh size $h_\bl=(h_{l_1},\cdots,h_{l_d}).$ Based on the grid $\Omega_\bl$, we denote  $I_\bl^\bj=\{\bx:x_m\in(h_mj_m,h_m(j_{m}+1)),m=1,\cdots,d\}$ as an elementary cell, and 
$$\bV_\bl^k:=\{{ \bv \in Q^k(I^{\bj}_{\bl})}, \,\,  \bzero \leq \bj  \leq 2^{\bl}-\bone \}= V_{l_1,x_1}^k\times\cdots\times  V_{l_d,x_d}^k$$
as the tensor-product piecewise polynomial space, where $Q^k(I^{\bj}_{\bl})$ represents the collection of polynomials of degree up to $k$ in each dimension on cell $I^{\bj}_{\bl}$. 
If we use equal mesh refinement of size $h_N=2^{-N}$ in each coordinate direction, the  grid and space will be denoted by $\Omega_N$ and $\bV_N^k$, respectively.  

For both the Alpert's and the interpolatory multiwavelets, we can define their   multidimensional version. For example, the space corresponding to Alpert's bases is 
$$\bW_\bl^k=W_{l_1,x_1}^k\times\cdots\times  W_{l_d,x_d}^k.$$
We can see that 
$
\bV_N^k=\bigoplus_{\substack{ |\bl|_\infty \leq N\\\bl \in \mathbb{N}_0^d}} \bW_\bl^k,
$
while the standard sparse grid approximation space   is
\begin{equation}
\label{eq:hiere_sg}
\hat{\bV}_N^k:=\bigoplus_{\substack{ |\bl|_1 \leq N\\\bl \in \mathbb{N}_0^d}}\bW_\bl^k \subset \bV_N^k.
\end{equation}
The dimension of $\hat{\bV}_N^k$ scales as $O((k+1)^d2^NN^{d-1})$ \cite{wang2016elliptic}, which is significantly less than that of $\bV_N^k$ with exponential dependence on $Nd$, hence the name ``sparse grid".  For the more general case, when we use adaptivity, the index choice for active elements is denoted by $(\bl,\bj)\in H,$ where the set $H$ will be determined adaptively according to some specified criteria.

Basis functions in multidimensions are also defined by tensor products. For example, for Alpert's multiwavelets,
\begin{equation}\label{eq:multidim-basis}
v^\bj_{\bi,\bl}(\bx) := \prod_{m=1}^d v^{j_m}_{i_m,l_m}(x_m),
\end{equation}
for $\bl \in \mathbb{N}_0^d$, $\bj \in B_\bl := \{\bj\in\mathbb{N}_0^d: \,\mathbf{0}\leq\bj\leq\max(2^{\bl-\mathbf{1}}-\mathbf{1},\mathbf{0}) \}$ and $\mathbf{1}\leq\bi\leq \bk+\mathbf{1}$. 
In our software, a general expression of the numerical solution $u_h$ is represented as  
\begin{equation}
    u_h(\bm, t) 
    = \sum_{\substack{(\bl,\bj)\in H, \\\mathbf{1}\leq\bi\leq \bk+\mathbf{1}}} c_{\bi, \bl}^{\bj}(t) v_{\bi, \bl}^{\bj}(\bx). 
\end{equation}
 $u_h$ is stored via a class hierarchy. First, at the lowest level, \verb|template <class T> VecMultiD| is constructed to store data, which can deal with tensors in any dimensions under a unified framework.
The class \verb|Element| stores  data related to each element  $(\bl,\bj)$ including the coefficients of Alpert's basis functions and interpolation basis functions. For example, the data member 
\verb|std::vector<VecMultiD<double>> Element::ucoe_alpt| stores the coefficients $c^\bj_{\bi,\bl}$ of Alpert's basis with \verb|VecMultiD<double>|, and the index of this \verb|std::vector| is to denote the unknown variables: for a scalar equation, the size of this vector is 1; for a system of equations, the size of this vector is the number of unknown variables. \verb|ucoe_alpt[s]| is a \verb|VecMultiD| of dimension $d$ and has the total number of DoF $(k+1)^d$, which corresponds to the $s$-th unknown variable in the system. Last, we have class \verb|DGSolution| at the top level to organize  the entire DG solution. The most important data member in this class is \verb|std::unordered_map<int, Element> dg|, which stores all the active elements $H$ and the associated Hash keys. Here, we use the C\verb|++| container \verb|std::unordered_map| that store elements formed by combing a key value (the hash key is an \verb|int| determined by the mesh level and the support index of the element) and a content value (\verb|Element|), allowing for fast retrieval of individual elements based on their keys.

\subsection{Adaptivity}

To realize the adaptivity of the scheme, we implement class \verb|DGAdapt|, which is derived from \verb|DGSolution|. There are two constants prescribed by the user for fine tuning adaptivity, namely the refinement threshold $\varepsilon$ (\verb|DGAdapt::eps|) and the coarsen threshold $\eta$ (\verb|DGAdapt::eta|). In the computation, we usually take $\eta = 0.1 \varepsilon.$ The adaptive procedure relies on two key member functions  \verb|DGAdapt::refine()| and \verb|DGAdapt::coarsen()|. Here, we provide the source code of \verb|DGAdapt::refine()| to illustrate the algorithm. 

\begin{lstlisting}
void DGAdapt::refine()
{
  // before refine, set new_add variable to be false in all elements
  set_all_new_add_false();
  for (auto & iter : leaf)
  {
    if (indicator_norm(*(iter.second)) > eps) // l2 norm
    {
      // loop over all its children
      const std::set<std::array<std::vector<int>,2>> index_chd_elem = 
      index_all_chd(iter.second->level, iter.second->suppt);
      for (auto const & index : index_chd_elem)
      {
        int hash_key = hash.hash_key(index);
        // if index is not in current child index set, then add it to dg solution
        if ( iter.second->hash_ptr_chd.find(hash_key) == iter.second->hash_ptr_chd.end())
        {
          Element elem(index[0], index[1], all_bas, hash);
          add_elem(elem);
        }                
      }
    }
  }
  check_hole();
  update_leaf();
  update_leaf_zero_child();
  update_order_all_basis_in_dgmap();
}
\end{lstlisting}

The concept of child and parent elements based on the hierarchical structure is defined as follows. If an element $(\bl',\bj')$ satisfies the conditions that: (1) there exists an integer $m$ such that $1\le m \le d$ and $\bl' = \bl+\bm{e}_m$, where $\bm{e}_m$ denotes the unit vector in the $m$-direction, and the support of $(\bl',\bj')$ is within the support of $(\bl,\bj)$; (2) $\|\bl'\|_\infty\le N$, then  it  is  called  a  child  element  of $(\bl,\bj)$. Accordingly, element $(\bl,\bj)$ is called a parent element of $(\bl',\bj')$. In \verb|DGAdapt|, member function \verb|index_all_chd(level,suppts)| returns a \verb|set| of indices $(\bl',\bj')$ of all the child elements of the input element with index \verb|(level,suppts)|, i.e., $(\bl,\bj)$. An element is called a leaf element if the number of its child elements  in \verb|dg| is less the maximum number of child elements it can have, i.e., at least one of its child elements are not included in \verb|dg|. \verb|DGAdapt| has data member \verb|std::unordered_map<int, Element*> leaf| that organizes the leaf elements of \verb|dg|. In \verb|DGAdapt::refine()|, we set variable \verb|new_add| to be false for all elements. Then we traverse the unordered map \verb|leaf| and compute the $\ell^2$ norm of Alpert wavelet coefficients $\left(\sum_{\mathbf{1} \leq \mathbf{i} \leq \mathbf{k}+\mathbf{1}}\left|c_{\bi, \bl}^{\bj}\right|^2\right)^{\frac{1}{2}}$ for each element  as the error indicator.  If it is larger than \verb|eps|, then we refine the element by adding all its child elements to \verb|dg|. The coefficients of the newly added elements are set to zero. Then \verb|check_hole()| is called to ensure that  all the parent elements of the newly added elements are in \verb|dg| (i.e., no “hole” is allowed). \verb|update_leaf()| will update map \verb|leaf|. Furthermore,   \verb|DGAdapt| has another data member \verb|std::unordered_map<int, Element*> leaf_zero_child|, which is a subset of \verb|leaf| and organizes the leaf elements with zero child elements. It plays a critical role in  function \verb|DGAdapt::coarsen()|. \verb|update_leaf_zero_child()| is called to update map \verb|leaf_zero_child|. Last, function \verb|update_order_all_basis_in_dgmap()| will update the ordering of all the basis in \verb|dg|, which will be used when assembling the operators. 
 
We also provide the source code of \verb|DGAdapt::coarsen()| to illustrate the coarsening algorithm.
\begin{lstlisting}	
void DGAdapt::coarsen()
{
  leaf.clear();
  update_leaf_zero_child();    
  coarsen_no_leaf();
  update_leaf();
  update_order_all_basis_in_dgmap();
}
\end{lstlisting}
In this routine, first we clear the map \verb|leaf| and update \verb|leaf_zero_child|. Then, we coarsen \verb|dg| based on map \verb|leaf_zero_child| by recursively calling member function \verb|coarsen_no_leaf()|. In particular, similar to the refining procedure, we traverse the map \verb|leaf_zero_child| and compute the $\ell^2$ norm of Alpert wavelet coefficients for each element. If it is less than \verb|eta|, then the element is removed from \verb|dg|. The coarsening procedure is repeatedly performed until no element can be removed. Lastly, we call \verb|update_leaf()| and \verb|update_order_all_basis_in_dgmap()| to update map \verb|leaf| and ordering of basis in the coarsened   \verb|dg|.

Now we have defined the basis class and the fundamental function approximation module. We can use it in any place when function approximations are needed, for example,  to initialize the PDE solution.
In the simple case when the initial condition $u_0=u_0(x_1, \cdots, x_d)$ is separable, i.e. $
	u_0(x_1, \cdots, x_d) = \prod_{i=1}^d g_i(x_i)$
or when it is a sum of several separable functions
$
	u_0(x_1, \cdots, x_d) = \sum_{k=1}^m \brac{\prod_{i=1}^d g_{ki}(x_i)},
$
we can first project each 1D function $g_i$ or $g_{ki}$ using numerical quadratures and then compute the coefficients of the basis functions in multidimensions by a tensor product. This is implement in the member functions \verb|DGAdapt::init_separable_scalar()| and \verb|DGAdapt::init_separable_scalar_sum()| for the scalar case and \verb|DGAdapt::init_separable_system()| and \verb|DGAdapt::init_separable_system_sum()| for the system case.
If the initial condition cannot be written in the separable form,   we will use the adaptive interpolation procedure by calling functions \verb|DGAdaptIntp::init_adaptive_intp_Lag()| or \verb|DGAdaptIntp::init_adaptive_intp_Herm()| corresponding to the Lagrange and Hermite interpolation. Then we perform the transformation to  the Alperts' multiwavelets by calling the function \verb|FastLagrInit::eval_ucoe_Alpt_Lagr()| with the fast matrix-vector multiplication in Section \ref{sec:fast}.  This is the main purpose of the class \verb|DGAdaptIntp|, which is derived from \verb|DGAdapt|.

%% file: operator.tex
\section{Basic operators}
\label{sec:operator}

In this section, we show the details of  implementing DG weak formulations in our package. We first introduce the operator matrix which stores the volume and interface interactions between all the basis functions in 1D. With this in hand, we can easily assemble the matrix for linear DG differential operators in arbitrary dimensions by recognizing the orthogonality of the Alpert's multiwavelets. Then, we introduce a fast matrix-vector multiplication algorithm, which play a critical role for computational savings when we perform the basis transformation. We also show the interpolation technique  with interpolatory multiwavelets for dealing with nonlinear terms together with the computation of bilinear forms. Last, we describe the  ODE solvers implmented in this package. The operator class described in this section as well as the adaptive procedure discussed in the previous section are essential steps for the adaptive evolution procedures in the algorithm.

\subsection{Operator matrix in 1D}

Unlike the standard DG method, for which each element can only interact with itself and its immediate neighbors, the interaction among multiwavelet basis is much more complicated when assembling DG bilinear forms due to the distinct hierarchical structures. Hence, it is critical to precompute and store the interaction information to save cost. In the package, the template class \verb|OperatorMatrix1D<class U, class V>| will compute and store the volume and interface interactions between all the basis functions in 1D. Here, the basis functions \verb|U| and \verb|V| can be Alpert basis (\verb|AlptBasis|), Lagrange interpolation basis (\verb|LagrBasis|) or Hermite interpolation basis (\verb|HermBasis|). The following code block shows how to use this class:
\begin{lstlisting}
// maximum mesh level
const int NMAX = 8;

// initialize all basis functions in 1D for Alpert, Lagrange and Hermite basis
AllBasis<AlptBasis> all_bas_alpt(NMAX);
AllBasis<LagrBasis> all_bas_lagr(NMAX);

// periodic boundary condition	
std::string boundary_type = "period";

OperatorMatrix1D<AlptBasis,AlptBasis> oper_matx_alpt(all_bas_alpt, all_bas_alpt, boundary_type);
OperatorMatrix1D<LagrBasis,AlptBasis> oper_matx_lagr(all_bas_lagr, all_bas_alpt, boundary_type);
\end{lstlisting}

In this code block, we first declare the maximum mesh level \verb|NMAX| to be 8. Then \verb|all_bas_alpt| and \verb|all_bas_lagr| store the information for all the Alpert basis and Lagrange interpolation basis functions, respectively. %
Under the periodic boundary condition, we compute and store the operators in \verb|oper_matx_alpt| and \verb|oper_matx_lagr|. For example, \verb|oper_matx_alpt.u_v| stores the inner product of all the basis functions in $L^2[0, 1]$. In particular, \verb|oper_matx_alpt.u_v(i,j)| denote the inner product of the $i$-th basis and the $j$-th basis, which forms an identity matrix due to the orthogonality.  \verb|oper_matx_alpt.u_vx(i,j)| denote the inner product of the $i$-th basis and the derivative of the $j$-th basis. We also provide operators involving the interface interactions. For example, \verb|oper_matx_alpt.ulft_vjp| stores $\sum_{i} u_{i+\half}^- [v]_{i+\half}$ where the summation is taken over all the cell interfaces where the basis function $v$ may have discontinuity.
These operator matrices are pre-computed at the beginning of the code. Since it only involves 1D calculation, the computational cost is negligible. %
We list all the operators in 1D in Table \ref{ex:table-operator-matrix-1D}.
\begin{table}[!hbp]
\centering
\caption{All the operators for the basis in 1D.}
\label{ex:table-operator-matrix-1D}
\begin{tabular}{c|c|c|c}
  \hline
 & variables & meaning & usage example \\
  \hline
\multirow{6}{4em}{volume}
& \verb|u_v| & $\int_0^1 uvdx$ & hyperbolic and HJ equations \\ 
& \verb|u_vx| & $\int_0^1 uv'dx$ & hyperbolic and ZK equations \\ 
& \verb|ux_v| & $\int_0^1 u'vdx$ & ZK equation \\ 
& \verb|ux_vx| & $\int_0^1 u'v'dx$ & diffusion and wave equations \\ 
& \verb|u_vxx| & $\int_0^1 uv''dx$ & ZK equations \\ 
& \verb|u_vxxx| & $\int_0^1 uv'''dx$ & KdV and ZK equations \\ 
  \hline
\multirow{10}{4em}{interface}
& \verb|ulft_vjp| & $\sum_{i} u_{i+\half}^- [v]_{i+\half}$ & hyperbolic and HJ equations \\ 
& \verb|urgt_vjp| & $\sum_{i} u_{i+\half}^+ [v]_{i+\half}$ & hyperbolic and HJ equations \\ 
& \verb|uxave_vjp| & $\sum_{i} \{u_x \}_{i+\half} [v]_{i+\half}$ & diffusion and wave equations \\ 
& \verb|ujp_vxave| & $\sum_{i} [u]_{i+\half} \{v_x \}_{i+\half}$ & diffusion and wave equations \\ 
& \verb|ujp_vjp| & $\sum_{i} [u]_{i+\half} [v]_{i+\half}$ & diffusion and wave equations \\ 
& \verb|uxxrgt_vjp| & $\sum_{i} (u_{xx})_{i+\half}^+ [v]_{i+\half}$ & KdV and ZK equations \\ 
& \verb|uxrgt_vxjp| & $\sum_{i} (u_x)_{i+\half}^+ [v_x]_{i+\half}$ & KdV and ZK equations \\ 
& \verb|ulft_vxxjp| & $\sum_{i} {u}_{i+\half}^- [v_{xx}]_{i+\half}$ & KdV and ZK equations \\ 
\hline
\end{tabular}
\end{table}
These operator matrix in 1D will be used in assembling the bilinear form in multidimensions in the DG scheme. We will show further details next.

\subsection{Bilinear form with Alpert multiwavelets}

The class \verb|BilinearFormAlpt| is the base class of assembling the bilinear form with Alpert multiwavelets. The DG bilinear forms resulted from linear equations with constant coefficients are all inherited from this class. To illustrate the main idea, we take the scalar linear hyperbolic equations with constant coefficients in 2D as an example:
\begin{equation}\label{eq:linear-hyperbolic-2D}
	u_t + u_{x_1} + u_{x_2} = 0.
\end{equation}
The corresponding DG scheme in the global formulation for solving \eqref{eq:linear-hyperbolic-2D} is
\begin{equation*}
\begin{split}
	\int_{\Omega} (u_h)_t \phi_h dx_1 dx_2 =& \sum_{i,j} \int_{I_{ij}} (u_h (\phi_h)_{x_1} + u_h (\phi_h)_{x_2}) dx_1 dx_2 - \sum_{i,j}\int_{I_j}(\hat{u_h}_{i+\frac{1}{2},j} (\phi_h)^-_{i+\frac{1}{2},j} - \hat{u_h}_{i-\frac{1}{2},j} (\phi_h)^+_{i-\frac{1}{2},j})dx_2 \\
	&- \sum_{i,j}\int_{I_i}(\hat{u_h}_{i,j+\frac{1}{2}} (\phi_h)^-_{i,j+\frac{1}{2}} - \hat{u_h}_{i,j-\frac{1}{2}} (\phi_h)^+_{i,j-\frac{1}{2}}) dx_1,
\end{split}
\end{equation*}
with the upwind flux:
\begin{equation}
	\hat{u_h}_{i+\frac{1}{2},j} = {u_h}_{i+\frac{1}{2},j}^-, \quad \hat{u_h}_{i,j+\frac{1}{2}} = {u_h}_{i,j+\frac{1}{2}}^-.
\end{equation}
With the periodic boundary condition, we can simplify the scheme:
\begin{equation}\label{eq:dg-linear-simplify}
	\int_{\Omega} (u_h)_t \phi_h dx_1 dx_2 = \sum_{i,j} \int_{I_{ij}} (u_h (\phi_h)_{x_1} + u_h (\phi_h)_{x_2}) dx_1 dx_2 + \sum_{i,j}\int_{I_j} {u_h}^-_{i+\frac{1}{2},j} [\phi_h]_{i+\frac{1}{2},j} dx_2 + \sum_{i,j}\int_{I_i}({u_h}^-_{i,j+\frac{1}{2}} [\phi_h]_{i,j+\frac{1}{2}} ) dx_1.
\end{equation}
with
\begin{equation}
	[\phi_h]_{i+\frac{1}{2},j} := (\phi_h)_{i+\frac{1}{2},j}^+ - (\phi_h)_{i+\frac{1}{2},j}^-, \quad [\phi_h]_{i,j+\frac{1}{2}} := (\phi_h)_{i,j+\frac{1}{2}}^+ - (\phi_h)_{i,j+\frac{1}{2}}^-.
\end{equation}

Next, we present the details on the computation of the volume integral $\sum_{i,j} \int_{I_{ij}} u_h (\phi_h)_{x_1} dx_1 dx_2$ in \eqref{eq:dg-linear-simplify}. The other terms in \eqref{eq:dg-linear-simplify} can be computed in similar ways.
Denote
\begin{equation}\label{eq:vol-integral-linear-simplify}
	B(u, \phi) := \sum_{i,j} \int_{I_{ij}} (u \phi_{x_1}) dx_1 dx_2.
\end{equation}
Our strategy to compute this bilinear form is to compute it for every basis function, i.e., assemble the matrix. After the matrix is assembled, we can directly compute the residual using the matrix-vector multiplication in the linear algebra package \verb|Eigen|.
By taking the solution $u$ and the test function $\phi$ in \eqref{eq:vol-integral-linear-simplify} to be $u = v_{\bi, \bl}^{\bj}(\bx) = v_{i_1, l_1}^{j_1}(x_1) v_{i_2, l_2}^{j_2}(x_2)$ and $\phi = v_{\bi', \bl'}^{\bj'}(\bx) = v_{i_1', l_1'}^{j_1'}(x_1) v_{i_2', l_2'}^{j_2'}(x_2)$, we obtain:
\begin{equation}\label{eq:dg-linear-simplify-RHS}
\begin{split}	
B(u, \phi) ={}& \sum_{i,j} \int_{I_{ij}} (u \phi_{x_1}) dx_1 dx_2 \\
={}& \int_{\Omega} \brac{v^{j_1}_{i_1,l_1}(x_1) v^{j_2}_{i_2,l_2}(x_2)} \frac{d}{dx_1}\brac{ v^{j_1'}_{i_1',l_1'}(x_1) v^{j_2'}_{i_2',l_2'}(x_2) } dx_1 dx_2 \\
={}& \brac{\int_0^1 v^{j_1}_{i_1,l_1}(x_1) \frac{d}{dx_1}( v^{j_1'}_{i_1',l_1'}(x_1)) dx_1} \delta_{i_2,i_2'} \delta_{j_2,j_2'}\delta_{l_2,l_2'} \\
 ={}& ( v^{j_1}_{i_1,l_1}, (v^{j_1'}_{i_1',l_1'})') \delta_{i_2,i_2'} \delta_{j_2,j_2'}\delta_{l_2,l_2'}
\end{split}
\end{equation}
Here, $(\cdot , \cdot)$ denotes the inner product in the unit interval $[0,1]$. From this, we can see that the bilinear form could be nonzero, only if $i_2=i_2'$, $j_2=j_2'$ and $l_2=l_2'$, i.e. the indices of the basis functions are the same in $x_2$-dimension. Moreover, to compute it, only the inner product of the Alpert basis and the derivative of the Alpert basis (i.e. the operator matrix in 1D \verb|oper_matx_alpt.u_vx|) will be used. This indicates that the computation of the matrix only depends on \emph{one} operator matrix in 1D and the specific dimension of the bilinear form. This is also true for equations in higher dimensions. With these properties, the matrix can be fast assembled by the function \verb|void BilinearFormAlpt::assemble_matrix_alpt()|. We note that all the classes to assemble the matrix for the  bilinear forms  are inherited from the base class \verb|BilinearFormAlpt| and call the function \verb|void BilinearFormAlpt::assemble_matrix_alpt()| at the lower level. The current package includes the derived classes \verb|DiffusionAlpt| and \verb|DiffusionZeroDirichletAlpt| for Laplace operator using interior penalty DG method \cite{arnold1982interior} with periodic and zero Dirichlet boundary conditions, and \verb|KdvAlpt|, \verb|ZKAlpt|, \verb|SchrodingerAlpt| for ultra weak DG methods for KdV equations \cite{cheng2008discontinuous}, ZK equations \cite{huang2022class} and Schr\"{o}dinger equations \cite{chen2019ultra}. It is easy to generalize to new weak formulations by following the similar line.

\subsection{Fast algorithm of matrix-vector multiplication}\label{sec:fast}

In this part, we describe the fast algorithm of the matrix-vector multiplication. Since the multiwavelet basis functions are hierarchical, the evaluation of the residual yields denser matrix than those obtained by standard local DG bases, if the interpolatory multiwavelets are applied for the nonlinear terms and no longer have the orthogonality. Efficient implementations are therefore essential to ensure that the computational cost is on par with element-wise implementation of traditional DG schemes. Our work extends the fast matrix-vector multiplication in \cite{shen2010efficient,shen2010sparse} to an adaptive index set.

\subsubsection{2D case}

For simplicity, we first show the main idea in 2D case while it can be easily generalized to arbitrary dimension in the next part. 

Consider the matrix-vector multiplication in this form:
\begin{equation}\label{eq:full-trans-x-y}
    f_{n_1,n_2} = \sum_{0\le n_1',n_2'\le N} f'_{n_1',n_2'} t_{n_1', n_1}^{(1)} t_{n_2',n_2}^{(2)}, \quad 0\le n_1,n_2\le N.
\end{equation}
Here $\bm{F}=\{ f_{n_1,n_2} \}_{0\le n_1,n_2\le N} \in \mathbb{R}^{(N+1)\times(N+1)}$ and $\bm{F}'=\{ f'_{n_1,n_2} \}_{0\le n_1,n_2\le N} \in \mathbb{R}^{(N+1)\times(N+1)}$ are matrices. $\bm{T}^{(i)}=\{ t_{n', n}^{(i)} \}_{0\le n',n\le N} \in \mathbb{R}^{(N+1)\times(N+1)}$ for $i=1,2$ denote the transformation matrix in 1D in $x_1$ and $x_2$ dimensions.

We can compute \eqref{eq:full-trans-x-y} dimension by dimension: first do the transformation in $x_1$ dimension:
\begin{equation}\label{eq:full-trans-x}
    g_{n_1,n_2'} = \sum_{0\le n_1'\le N} f'_{n_1',n_2'} t_{n_1',n_1}^{(1)}, \quad 0\le n_1,n_2'\le N.
\end{equation}
and then do the transformation in $x_2$ dimension:
\begin{equation}\label{eq:full-trans-y}
    f_{n_1,n_2} = \sum_{0\le n_2'\le N} g_{n_1,n_2'} t_{n_2',n_2}^{(2)}, \quad 0\le n_1,n_2\le N.
\end{equation}
Here, the intermediate matrix is denoted by $\{ g_{n_1,n_2} \}_{0\le n_1,n_2\le N}\in\mathbb{R}^{(N+1)\times(N+1)}$. It is easy to verify that \eqref{eq:full-trans-x}-\eqref{eq:full-trans-y} is equivalent to \eqref{eq:full-trans-x-y}. Moreover, if we first do the transformation in $x_2$ dimension:
\begin{equation}
    g_{n_1',n_2} = \sum_{0\le n_1'\le N} f'_{n_1',n_2'} t_{n_2',n_2}^{(2)}, \quad 0\le n_1,n_2'\le N,
\end{equation}
and then do the transformation in $x_1$ dimension:
\begin{equation}
    f_{n_1,n_2} = \sum_{0\le n_2'\le N} g_{n_1',n_2} t_{n_1',n_1}^{(1)}, \quad 0\le n_1,n_2\le N,
\end{equation}
which is also equivalent to \eqref{eq:full-trans-x-y}.

Next, we extend the same idea to the matrix-vector multiplication with different range of indices:
\begin{equation}\label{eq:sparse-trans-x-y}
    f_{n_1,n_2} = \sum_{0\le n_1'+n_2'\le N} f'_{n_1',n_2'} t_{n_1',n_1}^{(1)} t_{n_2',n_2}^{(2)}, \quad 0\le n_1+n_2\le N.
\end{equation}
Here, $f_{n_1,n_2}$ and $f'_{n_1,n_2}$ are only defined for $0\le n_1+n_2\le N$.

We try to compute \eqref{eq:sparse-trans-x-y} in the same way: first do the transformation in $x_1$ dimension:
\begin{equation}\label{eq:sparse-trans-x}
    g_{n_1,n_2'} = \sum_{0\le n_1'\le N-n_2'} f'_{n_1',n_2'} t_{n_1',n_1}^{(1)}, \quad 0\le n_1+n_2'\le N,
\end{equation}
and then do the transformation in $x_2$ dimension:
\begin{equation}\label{eq:sparse-trans-y}
    f_{n_1,n_2} = \sum_{0\le n_2'\le N-n_1} g_{n_1,n_2'} t_{n_2',n_2}^{(2)}, \quad 0\le n_1+n_2\le N.
\end{equation}
Here, the intermediate variable $g_{n_1,n_2}$  is only defined for $0\le n_1+n_2\le N$. 

By plugging \eqref{eq:sparse-trans-x} into \eqref{eq:sparse-trans-y}, we have
\begin{equation}\label{eq:sparse-trans-dim-by-dim-x-y}
    f_{n_1,n_2} = \sum_{0\le n_1'\le N-n_2', 0\le n_2'\le N-n_1} f'_{n_1',n_2'} t_{n_1',n_1}^{(1)} t_{n_2',n_2}^{(2)}, \quad 0\le n_1+n_2 \le N.
\end{equation}
Notice that the summation set $\{(n_1',n_2')\mid 0\le n_1'\le N-n_2', 0\le n_2'\le N-n_1\}$ in \eqref{eq:sparse-trans-dim-by-dim-x-y} is different from the original one $\{(n_1',n_2')\mid 0\le n_1'+n_2'\le N \}$ in \eqref{eq:sparse-trans-x-y}. Therefore, the two algorithms \eqref{eq:sparse-trans-x-y} and \eqref{eq:sparse-trans-x}-\eqref{eq:sparse-trans-y} are not equivalent. However, under the condition that $T^{(1)}$ is lower triangular matrix or $T^{(2)}$ is upper triangular matrix, \eqref{eq:sparse-trans-x-y} and \eqref{eq:sparse-trans-x}-\eqref{eq:sparse-trans-y} are equivalent. This will be illustrated and proved in Proposition \ref{prop:equivalence-matrix-vector-2D}.

We can also choose to do the transformation in $x_2$ dimension first:
\begin{equation}\label{eq:sparse-trans-first-y}
    g_{n_1',n_2} = \sum_{0\le n_2'\le N-n_1'} f'_{n_1',n_2'} t_{n_2',n_2}^{(2)}, \quad 0\le n_1'+n_2\le N,
\end{equation}
and then do the transformation in $x_1$ dimension:
\begin{equation}\label{eq:sparse-trans-second-x}
    f_{n_1,n_2} = \sum_{0\le n_2\le N-n_1'} g_{n_1',n_2} t_{n_1',n_1}^{(1)}, \quad 0\le n_1+n_2\le N.
\end{equation}
Similar as before, under the condition that $T^{(1)}$ is  upper triangular matrix or $T^{(2)}$ is lower triangular matrix, \eqref{eq:sparse-trans-x-y} and \eqref{eq:sparse-trans-first-y}-\eqref{eq:sparse-trans-second-x} are equivalent. This is also shown in Proposition \ref{prop:equivalence-matrix-vector-2D}.

\begin{prop}[equivalence of matrix-vector multiplication in 2D]\label{prop:equivalence-matrix-vector-2D}
\begin{enumerate}
	\item 
	Under any of the following two conditions:
	\begin{enumerate}
	\item	
	$T^{(1)}$ is lower triangular matrix, i.e., $t_{n_1', n_1}^{(1)}=0$ for $n_1'< n_1$;
	\item
	$T^{(2)}$ is upper triangular matrix, i.e., $t_{n_2', n_2}^{(2)}=0$ for $n_2'> n_2$;
	\end{enumerate}
	the matrix-vector multiplication in \eqref{eq:sparse-trans-x-y} is equivalent to first doing transformation in $x_1$ dimension in \eqref{eq:sparse-trans-x} and then in $x_2$ dimension in \eqref{eq:sparse-trans-y}.
	
	\item
	Under any of the following two conditions:
	\begin{enumerate}
	\item
	$T^{(1)}$ is upper triangular matrix, i.e., $t_{n_1', n_1}^{(1)}=0$ for $n_1'> n_1$;
	\item
	$T^{(2)}$ is lower triangular matrix, i.e., $t_{n_2', n_2}^{(2)}=0$ for $n_2'< n_2$;
	\end{enumerate}
	the matrix-vector multiplication in \eqref{eq:sparse-trans-x-y} is equivalent to first doing transformation in $x_2$ dimension in \eqref{eq:sparse-trans-first-y} and then in $x_1$ dimension in \eqref{eq:sparse-trans-second-x}.
\end{enumerate}
\end{prop}
\begin{proof}
We only give the proof for the first case in 1(a) and other cases can be proved along the same line.

	If $T^{(1)}$ is lower triangular matrix, i.e., $t_{n_1', n_1}^{(1)}=0$ for $n_1'< n_1$, then the summation set in \eqref{eq:sparse-trans-dim-by-dim-x-y} will be reduced to $\{(n_1',n_2')\mid 0\le n_1'\le N-n_2', 0\le n_2'\le N-n_1\} \cap \{(n_1',n_2')\mid n_1'\ge n_1\} = \{(n_1',n_2')\mid n_1 \le n_1'\le N, 0\le n_2'\le N-n_1'\}$. The original    summation set in \eqref{eq:sparse-trans-x-y} will be reduced to $\{(n_1',n_2')\mid 0\le n_1'+n_2'\le N \} \cap \{(n_1',n_2')\mid n_1'\ge n_1\} = \{(n_1',n_2')\mid n_1 \le n_1'\le N, 0\le n_2'\le N-n_1'\}$. Therefore, these two summation set are equivalent.
\end{proof}

Motivated by Proposition \ref{prop:equivalence-matrix-vector-2D}, we first decompose $T^{(1)}$ into:
\begin{equation}
    T^{(1)} = L^{(1)} + U^{(1)},
\end{equation}
where $L^{(1)}=\{ l_{n_1',n_1}^{(1)} \}$ and $U^{(1)}=\{ u_{n_1',n_1}^{(1)} \}$ are the lower and upper part of $T^{(1)}$, respectively. Note that the diagonal part of $T^{(1)}$ can be either put in $L^{(1)}$ or $U^{(1)}$. Then we decompose the summation \eqref{eq:sparse-trans-x-y} to:
\begin{equation}
    \begin{aligned}
        f_{n_1,n_2} &= \sum_{0\le n_1'+n_2'\le N} f'_{n_1',n_2'} (l_{n_1',n_1}^{(1)} + u_{n_1',n_1}^{(1)}) t_{n_2',n_2}^{(2)},  \\
        &= \sum_{0\le n_1'+n_2'\le N} f'_{n_1',n_2'} l_{n_1',n_1}^{(1)} t_{n_2',n_2}^{(2)} + \sum_{0\le n_1'+n_2'\le N} f'_{n_1',n_2'} t_{n_2',n_2}^{(2)} u_{n_1',n_1}^{(1)}.
    \end{aligned}    
\end{equation}
To compute the first term $\sum_{0\le n_1'+n_2'\le N} f'_{n_1',n_2'} l_{n_1',n_1}^{(1)} t_{n_2',n_2}^{(2)}$, we first do transformation in $x_1$ dimension as in \eqref{eq:sparse-trans-x} and then in $x_2$ dimension as in \eqref{eq:sparse-trans-y}. To compute the second term $\sum_{0\le n_1'+n_2'\le N} f'_{n_1',n_2'} t_{n_2',n_2}^{(2)} u_{n_1',n_1}^{(1)}$, we first doing transformation in $x_2$ dimension as in \eqref{eq:sparse-trans-first-y} and then in $x_1$ dimension as in \eqref{eq:sparse-trans-second-x}. Note that here we can also decompose $T^{(2)}$ into the lower and upper parts.

\subsubsection{Multidimensional case and adaptive sparse grid}

To generalize the 2D case to the multidimensional case and the adaptive sparse grid method, we consider the matrix-vector multiplication in the following form:
\begin{equation}\label{eq:adapt-multiD}
    f_{\bm{n}_1,\bm{n}_2,\dots,\bm{n}_d} = \sum_{(\bm{n}_1',\bm{n}_2',\dots,\bm{n}_d')\in G} f'_{\bm{n}_1',\bm{n}_2',\dots,\bm{n}_d'} t_{\bm{n}_1', \bm{n}_1}^{(1)} t_{\bm{n}_2',\bm{n}_2}^{(2)}\cdots t_{\bm{n}_d',\bm{n}_d}^{(d)}, \quad (\bm{n}_1,\bm{n}_2,\dots,\bm{n}_d) \in G,
\end{equation}
where $\bm{n}_i=(l_i,j_i,k_i)$ and $\bm{n}_i'=(l_i',j_i',k_i') \in\mathbb{N}_0^{3}$ for $i=1,2,\cdots,d$. In the adaptive sparse grid, $l_i$, $j_i$ and $k_i$ denotes the mesh level, the support index and the polynomial degree, in the $i$-th dimension, respectively.

For any two indices $\bm{n}_i = (l_i, j_i, k_i)$ and $\bm{n}_i' = (l_i', j_i', k_i')$, we define the order: 
\begin{enumerate}
	\item $\bm{n}_i \preceq \bm{n}_i'$ (or $\bm{n}_i \prec \bm{n}_i'$) if and only if $l_i \le l_i'$ (or $l_i < l_i'$);
	\item $\bm{n}_i \succeq \bm{n}_i'$ (or $\bm{n}_i \succ \bm{n}_i'$) if and only if $l_i \ge l_i'$ (or $l_i > l_i'$).
\end{enumerate}
Based on this order relation, we say
\begin{enumerate}
	\item $\bm{T}^{(i)} = \{ t_{\bm{n}_i', \bm{n}_i}^{(i)} \}$ is (strictly) lower triangular if and only if $t_{\bm{n}_i', \bm{n}_i}^{(i)} = 0$ for $\bm{n}_i' \prec \bm{n}_i$ ($\bm{n}_i' \preceq \bm{n}_i$);

	\item $\bm{T}^{(i)} = \{ t_{\bm{n}_i', \bm{n}_i}^{(i)} \}$ is (strictly) upper triangular if and only if $t_{\bm{n}_i', \bm{n}_i}^{(i)} = 0$ for $\bm{n}_i' \succ \bm{n}_i$ ($\bm{n}_i' \succeq \bm{n}_i$);
\end{enumerate}

We assume that $G$ satisfies the requirement that it is downward closed, i.e. for any basis function with some index in the set $G$, the index corresponding to the basis function in its parent element is also in $G$. Note that this requirement is enforced in our adaptive procedure.

Next, we try to do the transformation dimension by dimension as in the 2D case. We start with the transformation along the $x_1$ dimension:
\begin{equation}\label{eq:adapt-x1}
    g^{(1)}_{\bm{n}_1,\bm{n}_2',\dots,\bm{n}_d'} = \sum_{(\bm{n}_1',\bm{n}_2',\dots,\bm{n}_d')\in G} f'_{\bm{n}_1',\bm{n}_2',\dots,\bm{n}_d'} t_{\bm{n}_1', \bm{n}_1}^{(1)}, \quad (\bm{n}_1,\bm{n}_2',\dots,\bm{n}_d')\in G,
\end{equation}
then the $x_2$ dimension:
\begin{equation}\label{eq:adapt-x2}
    g^{(2)}_{\bm{n}_1,\bm{n}_2,\bm{n}_3'\dots,\bm{n}_d'} = \sum_{(\bm{n}_1,\bm{n}_2',\dots,\bm{n}_d')\in G} g^{(1)}_{\bm{n}_1,\bm{n}_2',\dots,\bm{n}_d'} t_{\bm{n}_2', \bm{n}_2}^{(2)}, \quad (\bm{n}_1,\bm{n}_2,\dots,\bm{n}_d')\in G,
\end{equation}
and all the way up to $x_d$ dimension:
\begin{equation}\label{eq:adapt-xd}
    f_{\bm{n}_1,\bm{n}_2,\bm{n}_3\dots,\bm{n}_d} = \sum_{(\bm{n}_1,\bm{n}_2,\dots,\bm{n}_{d-1},\bm{n}_d')\in G} g^{(d-1)}_{\bm{n}_1,\bm{n}_2,\dots,\bm{n}_{d-1},\bm{n}_d'} t_{\bm{n}_d', \bm{n}_d}^{(d)}, \quad (\bm{n}_1,\bm{n}_2,\dots,\bm{n}_d)\in G.
\end{equation}
We try to plug the first equation \eqref{eq:adapt-x1} into the second one \eqref{eq:adapt-x2} and all the way up to the last one \eqref{eq:adapt-xd}:
\begin{equation}
    f_{\bm{n}_1,\bm{n}_2,\dots,\bm{n}_d} = \sum_{(\bm{n}_1',\bm{n}_2',\dots,\bm{n}_d')\in S_f} f'_{\bm{n}_1',\bm{n}_2',\dots,\bm{n}_d'} t_{\bm{n}_1', \bm{n}_1}^{(1)} t_{\bm{n}_2',\bm{n}_2}^{(2)}\cdots t_{\bm{n}_d',\bm{n}_d}^{(d)}, \quad (\bm{n}_1,\bm{n}_2,\dots,\bm{n}_d) \in G,
\end{equation}
where the overall summation set is
\begin{align*}
    S_{{f}} := \{ (\bm{n}_1',\bm{n}_2',\cdots,\bm{n}_{d-1}',\bm{n}_d') \mid 
   & (\bm{n}_1',\bm{n}_2',\cdots,\bm{n}_{d-1}',\bm{n}_d')\in G, \dots, \\
   & (\bm{n}_1,\bm{n}_2',\cdots,\bm{n}_{d-1}',\bm{n}_d')\in G, \dots, 
    H(\bm{n}_1,\bm{n}_2,\cdots,\bm{n}_{d-1},\bm{n}_d')\in G \},	
\end{align*}
for $(\bm{n}_1,\bm{n}_2,\dots,\bm{n}_d)\in G$.
Generally, $S_f$ is only a subset of the orginal summation set 
\begin{equation}\label{eq:sum-set-original}
	S := \{ (\bm{n}_1',\bm{n}_2',\cdots,\bm{n}_{d-1}',\bm{n}_d') \mid (\bm{n}_1',\bm{n}_2',\cdots,\bm{n}_{d-1}',\bm{n}_d')\in G \}.
\end{equation}
However, if we assume that, for some integer $1\le k\le d$, $\bm{T}^{(i)}$ for $i=1,\dots,k-1$ are lower triangular and $\bm{T}^{(i)}$ for $i=k+1,\dots,d$ are upper triangular (or $\bm{T}^{(i)}$ for $i=1,\dots,k-1$ are lower triangular and $\bm{T}^{(i)}$ for $i=k+1,\dots,d$ are upper triangular), then these two constrains are equivalent, i.e., $S_f = S$. This is proved in the following proposition.

\begin{prop}\label{prop:equivalence-matrix-vector-multiD}
	Suppose that there exists some integer $1\le k\le d$ such that $\bm{T}^{(i)}$ for $i=1,\dots,k-1$ are lower triangular and $\bm{T}^{(i)}$ for $i=k+1,\dots,d$ are upper triangular (or $\bm{T}^{(i)}$ for $i=1,\dots,k-1$ are lower triangular and $\bm{T}^{(i)}$ for $i=k+1,\dots,d$ are upper triangular), then the algorithm \eqref{eq:adapt-multiD} are equivalent to the computation \eqref{eq:adapt-x1}-\eqref{eq:adapt-x2}-\eqref{eq:adapt-xd} dimension by dimension.
\end{prop}
\begin{proof}
We show the proof in the case of $d=3$ and $\bm{T}^{(1)}$ and ${\bm{T}^{(2)}}$ are both lower triangular matrix. The other cases follow the same line.

Since $\bm{T}^{(1)}$ and ${\bm{T}^{(2)}}$ are both lower triangular matrix, we have the extra constraint
\begin{equation}\label{eq:proof-extra-constrain-1}
    \bm{n}_1'\succeq \bm{n}_1, \quad \bm{n}_2'\succeq \bm{n}_2.
\end{equation}
Take any index $(\bm{n}_1',\bm{n}_2',\bm{n}_3')\in S$. Since $\bm{n}_1'\succeq \bm{n}_1$, there are only two cases: the first is that the basis function with the index $(\bm{n}_1,\bm{n}_2',\bm{n}_3')$ belongs to the parent (or parent's parent) element of $(\bm{n}_1',\bm{n}_2',\bm{n}_3')$ and thus $(\bm{n}_1,\bm{n}_2',\bm{n}_3')\in G$; the second is that the supports of two basis functions with the indices $(\bm{n}_1',\bm{n}_2',\bm{n}_3')$ and $(\bm{n}_1,\bm{n}_2',\bm{n}_3')$ have empty intersection and thus $t_{\bm{n}_1',\bm{n}_1}^{(1)} = 0$. 
Due to $\bm{n}_2'\succeq \bm{n}_2$, we have similar conclusions as the $x_1$ dimension. Therefore, either $S\subseteq S_f$ or the two summations are equivalent.
\end{proof}

Motivated by Proposition \ref{prop:equivalence-matrix-vector-multiD}, we can decompose the matrix in the first $(d-1)$ dimensions into the lower and upper parts:
\begin{equation}
    f_{\bm{n}_1,\bm{n}_2,\dots,\bm{n}_d} = \sum_{(\bm{n}_1',\bm{n}_2',\dots,\bm{n}_d')\in G} f'_{\bm{n}'} (l_{\bm{n}_1', \bm{n}_1}^{(1)}+u_{\bm{n}_1', \bm{n}_1}^{(1)})(l_{\bm{n}_2', \bm{n}_2}^{(2)}+u_{\bm{n}_2', \bm{n}_2}^{(2)})\cdots (l_{\bm{n}_{d-1}', \bm{n}_{d-1}}^{(d-1)}+u_{\bm{n}_{d-1}', \bm{n}_{d-1}}^{(d-1)})t_{\bm{n}_d',\bm{n}_d}^{(d)},
\end{equation}
for $(\bm{n}_1,\bm{n}_2,\dots,\bm{n}_d)\in G$.
Here $\bm{L}^{(i)} = \{ l_{{\bm{n}_i'}, \bm{n}_i}^{(i)} \}$ and $\bm{U}^{(i)} = \{ u_{{\bm{n}_i'}, \bm{n}_i}^{(i)} \}$ for $i=1,2,\cdots,d-1$ denote the lower and upper parts of $T^{(i)}$, respectively. Then, we  multipy every term out, which will result in $2^{d-1}$ terms for the summation. For each term, we can perform the multiplication dimension by dimension. The order is that first do the lower triangular matrix, and then the full matrix in between, and follows the upper triangular matrix. Here we list the case of $d=3$ as an example:
\begin{align*}
f_{\bm{n}_1,\bm{n}_2,\bm{n}_3} ={}& \sum_{(\bm{n}_1',\bm{n}_2',\bm{n}_3')\in G} f'_{\bm{n}'_1,\bm{n}_2',\bm{n}_3'} (l_{\bm{n}_1', \bm{n}_1}^{(1)}+u_{\bm{n}_2', \bm{n}_1}^{(1)})(l_{\bm{n}_2', \bm{n}_2}^{(2)}+u_{\bm{n}_2', \bm{n}_2}^{(2)})t_{\bm{n}_3',\bm{n}_3}^{(3)}, \\
={}& \sum_{(\bm{n}_1',\bm{n}_2',\bm{n}_3')\in G} f'_{\bm{n}_1',\bm{n}_2',\bm{n}_3'} l_{\bm{n}_1', \bm{n}_1}^{(1)} l_{\bm{n}_2',\bm{n}_2}^{(2)} t_{\bm{n}_3', \bm{n}_3}^{(3)} 
+ \sum_{(\bm{n}_1',\bm{n}_2',\bm{n}_3')\in G} f'_{\bm{n}_1',\bm{n}_2',\bm{n}_3'} l_{\bm{n}_1', \bm{n}_1}^{(1)} t_{\bm{n}_3',\bm{n}_3}^{(3)} u_{\bm{n}_2', \bm{n}_2}^{(2)} \\
& + \sum_{(\bm{n}_1',\bm{n}_2',\bm{n}_3')\in G} f'_{\bm{n}_1',\bm{n}_2',\bm{n}_3'} l_{\bm{n}_2', \bm{n}_2}^{(2)} t_{\bm{n}_3',\bm{n}_3}^{(3)} u_{\bm{n}_2', \bm{n}_2}^{(2)}
+ \sum_{(\bm{n}_1',\bm{n}_2',\bm{n}_3')\in G} f'_{\bm{n}_1',\bm{n}_2',\bm{n}_3'} t_{\bm{n}_3',\bm{n}_3}^{(3)} u_{\bm{n}_1', \bm{n}_1}^{(1)} u_{\bm{n}_2', \bm{n}_2}^{(2)}. \\
\end{align*}

The fast matrix-vector multiplication is implemented in \verb|void FastMultiplyLU::transform_1D()|, which do the multiplication dimension by dimension. The base class \verb|FastMultiplyLU| will be inherited by other classes and used in the interpolation, computational of the residuals. We will show in details in the later sections.

\subsection{Computing nonlinear terms}

The interpolation operator is applied when dealing with nonlinear terms. it is implemented in the base class \verb|Interpolation| and its inheritance class \verb|LagrInterpolation| and \verb|HermInterpolation|.
As an example, we consider scalar hyperbolic conservation laws in two dimension:
\begin{equation}
	u_t + f(u)_{x_1} + g(u)_{x_2} = 0.
\end{equation}
The semi-discrete DG scheme is
\begin{equation}
\begin{split}
	\int_{\Omega} (u_h)_t \phi dx_1 dx_2 =& \sum_{i,j} \int_{I_{ij}} (f(u_h) \phi_{x_1} + g(u_h) \phi_{x_2})dx_1 dx_2 - \sum_{i,j}\int_{I_j}(\hat{f}_{i+\frac{1}{2},j} \phi^-_{i+\frac{1}{2},j} - \hat{f}_{i-\frac{1}{2},j} \phi^+_{i-\frac{1}{2},j})dx_2 \\
	&- \sum_{i,j}\int_{I_i}(\hat{g}_{i,j+\frac{1}{2}} \phi^-_{i,j+\frac{1}{2}} - \hat{g}_{i,j-\frac{1}{2}} \phi^+_{i,j-\frac{1}{2}})dx_1.
\end{split}
\end{equation}
Here, we use the global Lax-Friedrichs flux:
\begin{equation}
	\hat{f}(u^-, u^+) = \frac{1}{2}(f(u^-)+f(u^+)) - \frac{\alpha}{2}(u^+-u^-).
\end{equation}
and
\begin{equation}
	\hat{g}(u^-, u^+) = \frac{1}{2}(g(u^-)+g(u^+)) - \frac{\alpha}{2}(u^+-u^-).
\end{equation}

In the classic DG methods, the integrals over elements and edges are often approximated by numerical quadrature rules on each cell \cite{DG4}. However, in the sparse grid DG method, this naive approach would result in computational cost that is proportional to the number of fundamental elements, i.e., $\mathcal{O}(h^{-d})$, and is still subject to the curse of dimensionality. To evaluate the integrals over elements and edges more efficiently with a cost proportional to the DOF of the underlying finite element space, we propose to interpolate the nonlinear function $f(u_h)$ by using the multiresolution Lagrange (or Hermite) interpolation basis functions \cite{huang2019adaptive}. Therefore, the semidiscrete DG scheme with interpolation is
\begin{equation}
\begin{split}
	\int_{\Omega} (u_h)_t \phi dxdy =& \sum_{i,j} \int_{I_{ij}} (\mathcal{I}[f(u_h)] \phi_{x_1} + \mathcal{I}[g(u_h)] \phi_{x_2})dx_1 dx_2 - \sum_{i,j}\int_{I_j}(\mathcal{I}[\hat{f}]_{i+\frac{1}{2},j} \phi^-_{i+\frac{1}{2},j} - \mathcal{I}[\hat{f}]_{i-\frac{1}{2},j} \phi^+_{i-\frac{1}{2},j})dx_2 \\
	&- \sum_{i,j}\int_{I_i}(\mathcal{I}[\hat{g}]_{i,j+\frac{1}{2}} \phi^-_{i,j+\frac{1}{2}} - \mathcal{I}[\hat{g}]_{i,j-\frac{1}{2}} \phi^+_{i,j-\frac{1}{2}})dx_1.
\end{split}
\end{equation}

Now we show the main procedure to interpolation $f(u_h)$. %
This is implemented in the function \verb|void HermInterpolation::nonlinear_Herm_2D_fast()| with the Hermite interpolation in 2D. The main procedure of the function is presented in the following:
\begin{lstlisting}
{
	fastHerm.eval_up_Herm();

	eval_fp_Her_2D(func, func_d1, func_d2, is_intp);

	pw1d.clear();

	eval_fp_to_coe_D_Her(is_intp);	
}
\end{lstlisting}
The first step in this function \verb|fastHerm.eval_up_Herm()| is to read the values (and derivatives) of $u_h$ at the interpolation points, i.e., transform the coefficients of Alpert basis in $u_h$ to the values (and derivatives) at the interpolation points. Here, the fast matrix-vector multiplication is used. The second step \verb|eval_fp_Her_2D(func, func_d1, func_d2, is_intp)| is to compute the values (and derivatives) of $f(u_h)$ at the interpolation points. This part is local in the sense that the computation of each point is independent of other points. Here, the value and derivatives of the scalar function $f(x)$ will be used. The last step \verb|eval_fp_to_coe_D_Her(is_intp)| is to transform the values (and derivatives) of $f(u_h)$ at the interpolation points to the coefficients of the interpolatory multiwavelets. Here, we use the fast matrix-vector multiplication again. The detailed algorithm is illustrated in \cite{tao2019collocation}.

Now we will show how to compute the bilinear form after the interpolation of $f(u_h)$ is expressed in terms of interpolatory multiwavelets. 
 As an example, we show the computation of the volume integral $\int_{\Omega} \mathcal{I}[f(u_h)] \phi_{x_1} dx_1 dx_2$. After the interpolation of $f(u_h)$, we obtain
\begin{equation}
	\mathcal{I}[f(u_h)] = \sum_{\substack{(\bl,\bj)\in H, \\\mathbf{1}\leq\bi\leq \bk+\mathbf{1}}} c_{\bi, \bl}^{\bj} \psi_{\bi, \bl}^{\bj}(\bx) = \sum_{\substack{(\bl,\bj)\in H, \\\mathbf{1}\leq\bi\leq \bk+\mathbf{1}}} c_{\bi, \bl}^{\bj} \psi_{i_1, l_1}^{j_1}(x_1) \psi_{i_2, l_2}^{j_2}(x_2)
\end{equation}
Take the test function $\phi = v_{\bi',\bl'}^{\bj'}(\bx) = v_{i_1',l_1'}^{j_1'}(x_1) v_{i_2',l_2'}^{j_2'}(x_2)$
\begin{align*}
	\int_{\Omega} \mathcal{I}[f(u_h)] \phi_{x_1} dx_1 dx_2 ={}& \int_{\Omega} \sum_{\substack{(\bl,\bj)\in H, \\\mathbf{1}\leq\bi\leq \bk+\mathbf{1}}} c_{\bi, \bl}^{\bj} \psi_{i_1, l_1}^{j_1}(x_1) \psi_{i_2, l_2}^{j_2}(x_2) \frac{d}{dx_1}v_{i_1',l_1'}^{j_1'}(x_1) v_{i_2',l_2'}^{j_2'}(x_2) dx_1 dx_2 \\	
	={}& \sum_{\substack{(\bl,\bj)\in H, \\\mathbf{1}\leq\bi\leq \bk+\mathbf{1}}}  c_{\bi, \bl}^{\bj} \int_{\Omega} \psi_{i_1, l_1}^{j_1}(x_1) \psi_{i_2, l_2}^{j_2}(x_2) \frac{d}{dx_1}v_{i_1',l_1'}^{j_1'}(x_1) v_{i_2',l_2'}^{j_2'}(x_2) dx_1 dx_2 \\
	={}& \sum_{\substack{(\bl,\bj)\in H, \\\mathbf{1}\leq\bi\leq \bk+\mathbf{1}}}  c_{\bi, \bl}^{\bj} \brac{\int_0^1 \psi_{i_1, l_1}^{j_1}(x_1) \frac{d}{dx_1}v_{i_1',l_1'}^{j_1'}(x_1) dx_1} \brac{\int_0^1 \psi_{i_2, l_2}^{j_2}(x_2)  v_{i_2',l_2'}^{j_2'}(x_2) dx_2}
\end{align*}
This can be efficiently computed using the fast matrix-vector multiplication in Section \ref{sec:fast}. In particular, the function 
\verb|void HyperbolicHermRHS::rhs_vol_scalar()| in the class \verb|HyperbolicHermRHS| is to compute this residual using the fast algorithm implemented in the base class \verb|FastRHS|.

\subsection{ODE solvers}

We implement the commonly used ODE solvers in the base class \verb|ODESolver| and its inheritance classes. Here, we use the linear algebra package \verb|Eigen| \cite{eigenweb} to perform matrix-vector multiplication and linear solvers. 
For the explicit RK method, the package includes the first-order Euler forward method in \verb|ForwardEuler|, the second-order and third-order strong-stability-preserving (SSP) RK method \cite{shu1988jcp} in \verb|RK2SSP| and \verb|RK3SSP|, the classic fourth-order RK method in \verb|RK4|. The explicit multistep method includes the second-order and fourth-order Newmark method \cite{wanner1996solving} in \verb|Newmark2nd| and \verb|Newmark4th|. The implicit-explicit (IMEX) method includes the third-order IMEX RK method \cite{pareschi2005implicit} in \verb|IMEX43|.

%% file: numerical.tex
\section{Examples}
\label{sec:example}

In this section, we use several examples to illustrate how to use this package and the code performance. Due to the page limit, we only present several representative equations including the linear equation with constant coefficients, the Hamilton-Jacobi equations, and the wave equations. There are also other examples available in the GitHub repository, such as the Schr\"{o}dinger equations, the Korteweg-de Vries (KdV) equation and its two-dimensional generalization, the Zakharov-Kuznetsov (ZK) equation. We will show the sample code and the CPU time. All the test cases are run in the High Performance Computing Center (HPCC) at Michigan State University with the AMD EPYC 7H12 Processor @2.595 GHz. The single node with the multi-thread (OpenMP) and the GCC 8.3.0 compiler is used.

\subsection{Linear hyperbolic equation with constant coefficients}
We consider the linear hyperbolic equation with constant coefficients in the domain $\Omega = [0, 1]^d$ with periodic boundary conditions:
\begin{equation}\label{eq:linear-eq-const-coefficient}
	u_t + \sum_{i=1}^d u_{x_i} = 0
\end{equation}
with the initial condition
\begin{equation}
	u(x_1,\cdots,x_d,0) = \cos(2\pi(\sum_{i=1}^d{x_i})).
\end{equation}

As the first example, we show the sparse grid DG method without adaptivity. The code is in the GitHub repository \verb|./example/02_hyperbolic_01_scalar_const_coefficient.cpp|. Here, we present the main part of the code for solving this problem. 

The first part of the code is to declare some basic parameters in the package including the dimension of the problem, the polynomial degrees of the Alpert basis and interpolation basis, the maximum mesh level, the CFL number and the final time and so on. Note that this code has uniform treatment with different dimensions. One can simply modify the dimension in this part and most of the functions in the package are consistent with arbitrary dimensions. However, we would also like to point out an implementation shortcoming of the current version of our package. Since many variables (e.g. \verb|HermBasis::PMAX| and \verb|HermBasis::msh_case|) in the package are declared as static variables in C\texttt{++}, they have to be declared even if they are not really used when solving the linear equations \eqref{eq:linear-eq-const-coefficient}.
\begin{lstlisting}
// --- Part 1: preliminary part	---
// static variables
const int DIM = 4;

AlptBasis::PMAX = 2;

LagrBasis::PMAX = 5;
LagrBasis::msh_case = 1;

HermBasis::PMAX = 3;
HermBasis::msh_case = 1;

Element::PMAX_alpt = AlptBasis::PMAX;	// max polynomial degree for Alpert basis
Element::PMAX_intp = LagrBasis::PMAX;	// max polynomial degree for interpolation basis
Element::DIM = DIM;			// dimension
Element::VEC_NUM = 1;		// num of unknown variables in PDEs

DGSolution::DIM = Element::DIM;
DGSolution::VEC_NUM = Element::VEC_NUM;

Interpolation::DIM = DGSolution::DIM;
Interpolation::VEC_NUM = DGSolution::VEC_NUM;

DGSolution::ind_var_vec = { 0 };
DGAdapt::indicator_var_adapt = { 0 };

Element::is_intp.resize(Element::VEC_NUM);
for (size_t num = 0; num < Element::VEC_NUM; num++)
{ Element::is_intp[num] = std::vector<bool>(Element::DIM, true); }

// constant variable
int NMAX = 4;
int N_init = NMAX;	
int is_sparse = 1;
const std::string boundary_type = "period";
double final_time = 1.0;
const double cfl = 0.1;
const bool is_adapt_find_ptr_alpt = true;	// variable control if need to adaptively find out pointers related to Alpert basis in DG operators
const bool is_adapt_find_ptr_intp = false;	// variable control if need to adaptively find out pointers related to interpolation basis in DG operators

// adaptive parameter
// if need to test code without adaptive, just set refine_eps large number 1e6, then no refine
// and set coarsen_eta negative number -1, then no coarsen
const double refine_eps = 1e10;
// const double coarsen_eta = refine_eps/10.;
const double coarsen_eta = -1;

OptionsParser args(argc, argv);
args.AddOption(&NMAX, "-N", "--max-mesh-level", "Maximum mesh level");
args.AddOption(&is_sparse, "-s", "--sparse-grid", "sparse grid (1) or full grid (0)");
args.AddOption(&final_time, "-tf", "--final-time", "Final time; start time is 0.");
args.Parse();
if (!args.Good())
{
	args.PrintUsage(std::cout);
	return 1;
}
args.PrintOptions(std::cout);

N_init = NMAX;
bool sparse = (is_sparse == 1) ? true : false;

// hash key
Hash hash;

LagrBasis::set_interp_msh01();
HermBasis::set_interp_msh01();

AllBasis<AlptBasis> all_bas_alpt(NMAX);
AllBasis<LagrBasis> all_bas_lagr(NMAX);
AllBasis<HermBasis> all_bas_herm(NMAX);	

// operator matrix
OperatorMatrix1D<AlptBasis,AlptBasis> oper_matx(all_bas_alpt, all_bas_alpt, boundary_type);
OperatorMatrix1D<HermBasis, HermBasis> oper_matx_herm_herm(all_bas_herm, all_bas_herm, boundary_type);
OperatorMatrix1D<LagrBasis, LagrBasis> oper_matx_lagr_lagr(all_bas_lagr, all_bas_lagr, boundary_type);
\end{lstlisting}

Part 2 of the code is to project the initial condition to the sparse grid DG space. Here, although the initial function is low rank,  we still use the function \verb|DGAdaptIntp::init_adaptive_intp()|. for a general solution. The first step in this function is to do adaptive Lagrange interpolation and update coefficients of interpolation basis (\verb|Element::ucoe_intp| in DG solution). The second step is to transform coefficients of Lagrange basis to those of Alpert basis where the fast matrix-vector multiplication is applied.
\begin{lstlisting}	
// --- Part 2: initialization of DG solution ---
DGAdaptIntp dg_solu(sparse, N_init, NMAX, all_bas_alpt, all_bas_lagr, all_bas_herm, hash, refine_eps, coarsen_eta, is_adapt_find_ptr_alpt, is_adapt_find_ptr_intp, oper_matx_lagr_lagr, oper_matx_herm_herm);

// initial condition
// u(x,0) = cos(2 * pi * (sum_(d=1)^DIM x_d)
auto init_non_separable_func = [=](std::vector<double> x, int i)
{	
	double sum_x = 0.;
	for (int d = 0; d < DIM; d++) { sum_x += x[d]; };

	return cos(2*Const::PI*sum_x);
};

dg_solu.init_adaptive_intp(init_non_separable_func);
\end{lstlisting}

After the initial condition is projected, we do the time evolution. First we assemble the matrix for the DG bilinear form using the function \verb|void HyperbolicAlpt::assemble_matrix_scalar()|. Then we use third-order SSP RK time integrator \cite{shu1988jcp} to do the time marching. 
\begin{lstlisting}
// --- Part 3: time evolution ---
// coefficients in the equation are all 1:
// u_t + \sum_(d=1)^DIM u_(x_d) = 0
const std::vector<double> hyperbolicConst(DIM, 1.);

const int max_mesh = dg_solu.max_mesh_level();
const double dx = 1./pow(2., max_mesh);
double dt = dx * cfl / DIM;
int total_time_step = ceil(final_time/dt) + 1;
dt = final_time/total_time_step;

HyperbolicAlpt dg_operator(dg_solu, oper_matx);
dg_operator.assemble_matrix_scalar(hyperbolicConst);

RK3SSP odesolver(dg_operator, dt);
odesolver.init();

std::cout << "--- evolution started ---" << std::endl;
// record code running time
Timer record_time;
double curr_time = 0;
for (size_t num_time_step = 0; num_time_step < total_time_step; num_time_step++)
{	
	odesolver.step_rk();
	curr_time += dt;
	
	// record code running time
	if (num_time_step %
	{		
		std::cout << "num of time steps: " << num_time_step 
				<< "; time step size: " << dt 
				<< "; curr time: " << curr_time
				<< "; DoF: " << dg_solu.get_dof()
				<< std::endl;
		record_time.time("elasped time in evolution");
	}
}	
odesolver.final();
\end{lstlisting}

The final part is to compute the $L^2$ error between the numerical solution and the exact solution. Since the evaluation of the error using the Gaussian quadrature in high dimension is too costly, we again use adaptive interpolation to approximate the exact solution at final time, then transform them to the coefficients of Alpert basis.
\begin{lstlisting}
// --- Part 4: calculate error between numerical solution and exact solution ---
std::cout << "calculating error at final time" << std::endl;
record_time.reset();

// compute the error using adaptive interpolation
{
	// construct anther DGsolution v_h and use adaptive Lagrange interpolation to approximate the exact solution
	const double refine_eps_ext = 1e-6;
	const double coarsen_eta_ext = -1; 
	OperatorMatrix1D<HermBasis, HermBasis> oper_matx_herm_herm(all_bas_herm, all_bas_herm, boundary_type);
	OperatorMatrix1D<LagrBasis, LagrBasis> oper_matx_lagr_lagr(all_bas_lagr, all_bas_lagr, boundary_type);
	
	DGAdaptIntp dg_solu_ext(sparse, N_init, NMAX, all_bas_alpt, all_bas_lagr, all_bas_herm, hash, refine_eps_ext, coarsen_eta_ext, is_adapt_find_ptr_alpt, is_adapt_find_ptr_intp, oper_matx_lagr_lagr, oper_matx_herm_herm);

	auto final_func = [=](std::vector<double> x, int i) 
	{	
		double sum_x = 0.;
		for (int d = 0; d < DIM; d++) { sum_x += x[d]; };
		return cos(2.*Const::PI*(sum_x - DIM * final_time));
	};
	dg_solu_ext.init_adaptive_intp(final_func);

	// compute L2 error between u_h (numerical solution) and v_h (interpolation to exact solution)
	double err_l2 = dg_solu_ext.get_L2_error_split_adaptive_intp_scalar(dg_solu);		
	std::cout << "L2 error at final time: " << err_l2 << std::endl;	
}			
\end{lstlisting}

We report the convergence study of the sparse grid method at $t=1$ in Table \ref{tb:linear-hyperbolic}, including the L2 errors and the associated orders of accuracy. It is observed that the sparse grid method has about half-order reduction from the optimal $(k + 1)$-th order for high-dimensional computations, which is expected from our analysis \cite{guo2016transport}.

To further demonstrate the efficiency of the sparse grid algorithm, we report the $L^2$ errors versus the average CPU cost per time step for $k = 1, 2$ and $d = 2, 3, 4$ in Fig. \ref{fig:linear-cpu-err}. It is observed that, to achieve a desired level of accuracy, the sparse grid DG method with a larger $k$ requires less CPU time as expected. Moreover, the CPU time is approximately proportional to the DOF. In addition, the slopes for all the lines in Fig. \ref{fig:linear-cpu-err} are approximately the same with different dimensions. This indicates that our method has potential in the computations in high dimensions.

\begin{table}[!hbp]
\centering
\caption{linear hyperbolic equation with constant coefficients, $L^2$-error and convergence order at $t=1$.}
\label{tb:linear-hyperbolic}
\begin{tabular}{c|c|c|c|c|c|c|c|c}
\hline
\multicolumn{1}{c|}{} &  \multicolumn{4}{c|}{$k=1$} &\multicolumn{4}{c}{$k=2$}\\
\hline
& $N$ & DOF & $L^2$-error & order & $N$ & DOF & $L^2$-error  & order \\
\hline
\multirow{5}{3em}{$d=2$}
& 5 & 448   & 1.59e-2 & -    & 4 & 432   & 1.50e-3 & - \\
& 6 & 1024  & 3.84e-3 & 2.05 & 5 & 1008  & 3.95e-4 & 1.93 \\
& 7 & 2304  & 9.80e-4 & 1.97 & 6 & 2304  & 3.54e-5 & 3.48 \\
& 8 & 5120  & 2.75e-4 & 1.84 & 7 & 5184  & 6.70e-6 & 2.40 \\
& 9 & 11264 & 7.33e-5 & 1.91 & 8 & 11520 & 7.10e-7 & 3.24 \\
\hline
\multirow{5}{3em}{$d=3$}
& 4 & 832   & 4.60e-1 & -	 & 4 & 2808   & 8.92e-3 & - \\
& 5 & 2176  & 1.48e-1 & 1.63 & 5 & 7344   & 1.66e-3 & 2.42 \\
& 6 & 5504  & 3.30e-2 & 2.16 & 6 & 18576  & 3.85e-4 & 2.11 \\
& 7 & 13568 & 7.55e-3 & 2.13 & 7 & 45792  & 4.29e-5 & 3.17 \\
& 8 & 32768 & 2.15e-3 & 1.82 & 8 & 110592 & 1.07e-5 & 2.00 \\
\hline
\multirow{5}{3em}{$d=4$}
& 5 & 8832   & 5.19e-1 & -	  & 4 & 15552  & 3.44e-2 & - \\
& 6 & 24320  & 2.12e-1 & 1.29 & 5 & 44712  & 6.08e-3 & 2.50 \\
& 7 & 64768  & 5.50e-2 & 1.94 & 6 & 123120 & 1.20e-3 & 2.35 \\
& 8 & 167936 & 1.24e-2 & 2.15 & 7 & 327888 & 2.77e-4 & 2.11 \\
& 9 & 425984 & 3.91e-3 & 1.66 & 8 & 850176 & 5.90e-5 & 2.23 \\
\hline
\end{tabular}
\end{table}

\begin{figure}
    \centering
    \includegraphics[width=0.7\textwidth]{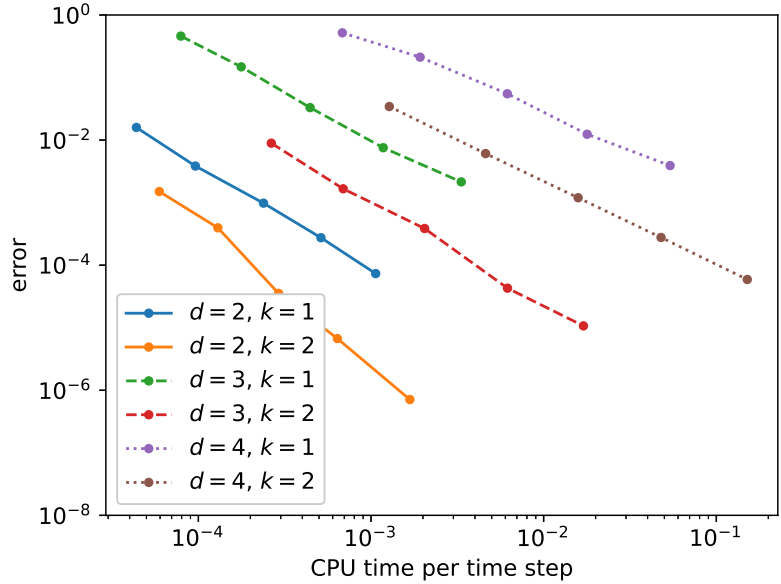}
    \caption{linear hyperbolic equation with constant coefficients. $L^2$-error versus CPU time per time step. Sparse grid DG. $t=1$}
    \label{fig:linear-cpu-err}
\end{figure}

\subsection{Hamilton-Jacobi equation}

In this example, we consider the following Hamilton-Jacobi-Bellman (HJB) equation \cite{bokanowski2013adaptive}
\begin{equation}\label{ex:HJB_ex}
\left\{\begin{array}{l}\displaystyle
\phi_t +\max_{\bb\in\mathcal{B}}\left(\sum_{m=1}^d b_m\cdot\nabla\phi\right) = 0,\quad \bx\in[0,1]^d,\\
\phi(\bx, 0)=g(\|\bx-\ba\|),
\end{array}\right.
\end{equation}	
where $\ba=(0.5,0.5,\ldots,0.5)$ and $\mathcal{B} = \{\bb=(b_1,b_2,\ldots,b_d),\,b_m=\pm1,\,m=1,\cdots,d \} $ is a set of $2^d$ vectors corresponding to $2^d$ possible controls. 
The function $g(z)$ is given by
\begin{equation}
	g(z) = \frac{1}{r_0}(z^2 - r_0^2),
\end{equation}
with $r_0 = \frac{1}{8}$.
Note that this HJB equation is equivalent to the following HJ equation
\begin{equation}
\label{ex:HJ_ex}
\left\{\begin{array}{l}\displaystyle
\phi_t + \sum_{m=1}^d|\phi_{x_m}| = 0,\quad \bx\in[0,1]^d\\
\phi(\bx, 0)=g(\|\bx-\ba\|),
\end{array}\right.
\end{equation}	
The exact solution can be hence derived from \eqref{ex:HJ_ex}:
$$
\phi(\bx, t)=g(\|(\bx-\ba))_t^\star\|).
$$ 
Here, for a vector $\mathbf{c}$, $\mathbf{c}_t^\star: = \min (\max (0, \mathbf{c}-t), \mathbf{c}+t)$ in the component-wise sense.
We apply the adaptive algorithm to simulate \eqref{ex:HJ_ex}. The outflow boundary conditions are imposed. Note that the Hamiltonian is nonsmooth and in \cite{guo2021adaptive} we regularize the absolute function to ensure stability:
\begin{equation}
\label{eikonal_regu}
\tilde H(\nabla\phi) = \begin{cases}
H(\nabla\phi), &\text{if } \|\nabla\phi\| \ge \delta  \\
\frac{1}{2\delta} H(\nabla\phi)^2 + \frac12\delta, &\text{otherwise}.
\end{cases}
\end{equation}
It can be easily verified that $\tilde H$ is $C^1$. In the simulation, we choose $\delta = 2h$, where $h$ is the mesh size, and hence the regularization will not affect the accuracy of the original method.

This numerical test has been already examined in \cite{guo2021adaptive} by a local DG (LDG) method and here, we only present the sample code and focus on the CPU time study. The   code is included in the GitHub repository \verb|./example/04_hamilton_jacobi_adapt_hjb.cpp|. Due to the page limit, we will only present the time evolution code. 

The first part in the time evolution is to compute the time step size based on the finest mesh size in all the dimensions. Note that we use adaptive finite element space, so the finest mesh size is changing during the time evolution. The second part is to use the Euler forward time stepping to predict the numerical solution in the next time step. Here we use the LDG method in \cite{yan2011local} to reconstruct the first-order derivatives of $\phi$, i.e., $\phi_{x_m}$, $m = 1, \cdots, d$. In particular, the LDG method computes two auxiliary variables in each dimension, which approximate the first-order derivatives with opposite one-sided numerical fluxes. Therefore, there are totally $(2d+1)$ unknown variables in the scheme including the unknown $\phi$ and the other $2d$ auxiliary variables. For the $2d$ auxiliary variables, the LDG bilinear forms are linear, and we evolve them by assembling the matrices stored in \verb|grad_linear| in the code. For the evolution of $\phi$, we first use Lagrangian interpolation basis to approximate the nonlinear Hamiltonian \verb|nonlinear_Lagr_fast| and then apply fast matrix-vector multiplication.  Since the assembling matrix is independent for each auxiliary variable, we use OpenMP here to improve the computational efficiency. After this, we call the function  and then \verb|rhs_nonlinear| to evaluate the residual (i.e. the right-hand-side of the weak formulation of $\phi$ and then call the time integrator to update the solution $\phi$ in the next time step.

The third part is to refine the numerical solution based on the prediction. After this refinement, the fourth part is to do the time evolution using the third-order SSP RK method \cite{shu1988jcp}. Note that this part is similar to the prediction part by calling the same function. The only difference is to use \verb|RK3SSP| instead of \verb|ForwardEuler|. In the end of the time evolution, we do coarsening to remove the redundant elements.

\begin{lstlisting}
while (curr_time < final_time)
{

	auto start_evolution_time = std::chrono::high_resolution_clock::now();

	// --- part 1: calculate time step dt ---
	const std::vector<int> & max_mesh = dg_solu.max_mesh_level_vec();

	// dt = cfl/(c1/dx1 + c2/dx2 + ... + c_dim/dx_dim)
	double sum_c_dx = 0.;	// this variable stores (c1/dx1 + c2/dx2 + ... + c_dim/dx_dim)
	for (size_t d = 0; d < DIM; d++)
	{
		sum_c_dx += std::pow(2., max_mesh[d]);
	}
	double dt = cfl / sum_c_dx;
	dt = std::min(dt, final_time - curr_time);

	// --- part 2: predict by Euler forward ---
	{
		std::vector<HJOutflowAlpt> grad_linear(2 * DIM, HJOutflowAlpt(dg_solu, oper_matx_alpt, oper_matx_alpt_inside, 1));

		omp_set_num_threads(2 * DIM);
#pragma omp parallel for
		for (int d = 0; d < 2 * DIM; ++d) // 2 * DIM matrices for computing the gradient of phi via LDG
		{
			int sign = 2 * (d %
			int dd = d / 2;
			grad_linear[d].assemble_matrix_flx_scalar(dd, sign, -1);
			grad_linear[d].assemble_matrix_vol_scalar(dd, -1);
		}
		HamiltonJacobiLDG HJ(grad_linear, DIM);

		// before Euler forward, copy Element::ucoe_alpt to Element::ucoe_alpt_predict
		dg_solu.copy_ucoe_to_predict();

		ForwardEuler odeSolver(dg_solu, dt);
		odeSolver.init_HJ(HJ);

		dg_solu.set_rhs_zero();

		odeSolver.compute_gradient_HJ(HJ); // compute the gradient of phi

		interp.nonlinear_Lagr_fast(LFHamiltonian, is_intp, fastintp); // interpolate the LFHamiltonian

		fastRHShj.rhs_nonlinear();

		odeSolver.rhs_to_eigenvec("HJ");

		odeSolver.step_stage(0);

		odeSolver.final_HJ(HJ);
	}

	// --- part 3: refine base on Element::ucoe_alpt
	dg_solu.refine();
	const int num_basis_refine = dg_solu.size_basis_alpt();

	// after refine, copy Element::ucoe_alpt_predict back to Element::ucoe_alpt
	dg_solu.copy_predict_to_ucoe();

	// --- part 4: time evolution
	std::vector<HJOutflowAlpt> grad_linear(2 * DIM, HJOutflowAlpt(dg_solu, oper_matx_alpt, oper_matx_alpt_inside, 1));
	omp_set_num_threads(2 * DIM);
#pragma omp parallel for
	for (int d = 0; d < 2 * DIM; ++d) // 2 * DIM matrices for computing the gradient of phi via LDG
	{
		int sign = 2 * (d %
		int dd = d / 2;
		grad_linear[d].assemble_matrix_flx_scalar(dd, sign, -1);
		grad_linear[d].assemble_matrix_vol_scalar(dd, -1);
	}
	HamiltonJacobiLDG HJ(grad_linear, DIM);

	RK3SSP odeSolver(dg_solu, dt);
	odeSolver.init_HJ(HJ);

	for (int stage = 0; stage < odeSolver.num_stage; ++stage)
	{
		dg_solu.set_rhs_zero();

		odeSolver.compute_gradient_HJ(HJ); // compute the gradient of phi

		interp.nonlinear_Lagr_fast(LFHamiltonian, is_intp, fastintp); // interpolate the LFHamiltonian
		
		fastRHShj.rhs_nonlinear();

		odeSolver.rhs_to_eigenvec("HJ");

		odeSolver.step_stage(stage);

		odeSolver.final_HJ(HJ);
	}

	// --- part 5: coarsen
	dg_solu.coarsen();
	const int num_basis_coarsen = dg_solu.size_basis_alpt();

	curr_time += dt;
}	
\end{lstlisting}

In this example, the viscosity solution is $C^1$, a rarefaction wave opens up at the center of the domain, which is well captured by the adaptive sparse method, see the solution profile in \cite{guo2021adaptive}. In Table \ref{tb:hjb}, we summarize the convergence study for $d = 2, 3, 4$ and $k = 1, 2$. Note that when $\epsilon=10^{-6}$, the error does not decay, and the reason is that the error has saturated already with the maximum level $N = 7$. In Figure \ref{fig:hj-cpu-err}, we plot the $L^2$ errors versus the average CPU cost per time step. It is again observed that, the errors with larger polynomial degree $k$ decays faster when reducing the refinement threshold $\epsilon$ and thus increasing DOF. The slopes with the same polynomial degrees in Figure \ref{fig:hj-cpu-err} are almost the same with different dimensions. This again domenstrates the capability of our method in high dimensions.

\begin{table}[!hbp]
	\centering
	\caption{Hamilton-Jacobi equation, $d=2,\,3,\,4$, $k=1,\,2$. Adaptive sparse grid. $T=0.1$ and the maximum mesh level $N=7$.}
	\label{tb:hjb}
	\begin{tabular}{c|c|c|c|c|c|c|c|c|c}
		\hline
		\multicolumn{2}{c|}{} &  \multicolumn{4}{c|}{$k=1$} &\multicolumn{4}{c}{$k=2$}\\
		\hline
		& $\epsilon$ & DOF & L$^2$-error  & $R_{\epsilon}$ & $R_{\textrm{DOF}}$& DOF & L$^2$-error  & $R_{\epsilon}$ & $R_{\textrm{DOF}}$\\
		\hline
		\multirow{4}{3em}{$d=2$}
&	1e-3	&	204	&	4.17e-3	&	-	&	-	&	63	&	8.62e-3	&	-	&	-	\\
&	1e-4	&	444	&	1.62e-3	&	0.41	&	1.21	&	135	&	1.89e-3	&	0.66	&	1.99	\\
&	1e-5	&	860	&	7.01e-4	&	0.36	&	1.27	&	207	&	6.26e-4	&	0.48	&	2.59	\\
&	1e-6	&	876	&	6.43e-4	&	0.04	&	4.69	&	459	&	4.24e-4	&	0.17	&	0.49	\\
		\hline
		\multirow{4}{3em}{$d=3$}
&	1e-3	&	608	&	5.44e-3	&	-	&	-	&	270	&	1.12e-2	&	-	&	-	\\
&	1e-4	&	1328	&	2.12e-3	&	0.41	&	1.20	&	594	&	2.98e-3	&	0.58	&	1.68	\\
&	1e-5	&	2576	&	9.15e-4	&	0.37	&	1.27	&	918	&	7.89e-4	&	0.58	&	3.05	\\
&	1e-6	&	2624	&	8.39e-4	&	0.04	&	4.74	&	2052	&	4.36e-4	&	0.26	&	0.74	\\
		\hline
		\multirow{4}{3em}{$d=4$} 
&	1e-3	&	1616	&	6.65e-3	&	-	&	-	&	1053	&	1.35e-2	&	-	&	-	\\
&	1e-4	&	3536	&	2.60e-3	&	0.41	&	1.20	&	1701	&	6.17e-3	&	0.34	&	1.63	\\
&	1e-5	&	6864	&	1.12e-3	&	0.37	&	1.27	&	2997	&	1.24e-3	&	0.70	&	2.84	\\
&	1e-6	&	6992	&	1.02e-3	&	0.04	&	4.69	&	6237	&	4.90e-4	&	0.40	&	1.26	\\
		\hline
	\end{tabular}
\end{table}	

\begin{figure}
    \centering
    \includegraphics[width=0.7\textwidth]{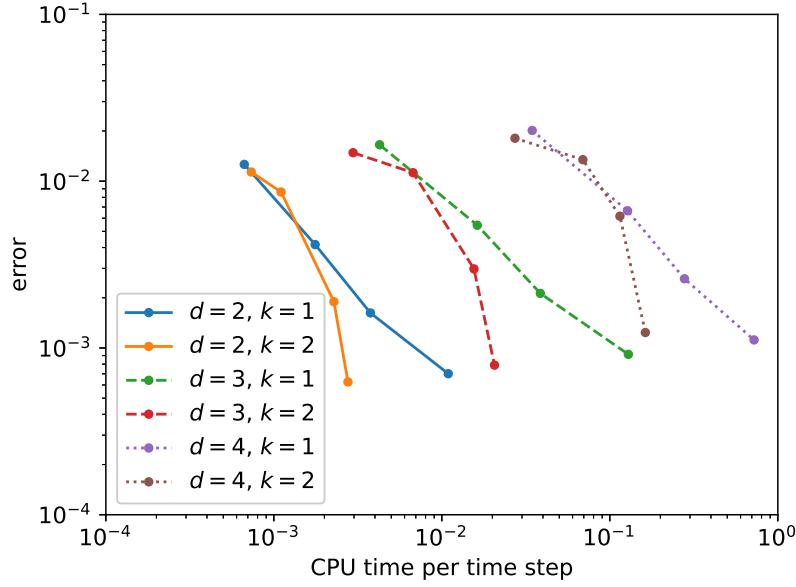}
    \caption{Hamilton-Jacobi equation. $L^2$-error versus CPU time per time step. Sparse grid DG. $t=1$}
    \label{fig:hj-cpu-err}
\end{figure}

\subsection{Wave equation}

In this example, we consider the isotropic wave propagation in heterogeneous media \cite{chou2014optimal} in the domain $[0,1]^d$ with periodic boundary conditions:
\begin{equation}\label{eq-wave}
u_{tt} =  \nabla \cdot (c^2(\bx) \nabla u ).
\end{equation}

For 2D case, the domain $\Omega = [0, 1]^2$ is composed of two subdomains $\Omega_1 = [\frac{1}{4}, \frac{3}{4}] \times [0, 1]$ and $\Omega_2 = \Omega \backslash \Omega_1$. The coefficient $c^2$ is a constant in each subdomain:
\begin{align}\label{example3_2_new}
c^2=
\begin{cases}
1, \qquad \text{in} \quad \Omega_1, \\
{\frac{5}{37}}, \qquad \text{in} \quad \Omega_2&
\end{cases}
\end{align}
With this setup, the exact solution is a standing wave
\begin{align}\label{example3_3_new}
u= 
\begin{cases}
\sin(\sqrt{20}\pi t)\cos(4\pi x_1) \cos(2\pi x_2 ), \qquad \text{in} \quad \Omega_1, \\
\sin(\sqrt{20}\pi t)\cos(12\pi x_1 )\cos(2\pi x_2), \qquad \text{in} \quad \Omega_2.
\end{cases}
\end{align}

For 3D case, $\Omega_1 = [\frac{1}{4}, \frac{3}{4}] \times [0, 1] \times [0, 1]$ and $\Omega_2 = \Omega \backslash \Omega_1$
\begin{align}\label{example3_2_3d}
c^2=
\begin{cases}
1, \qquad \text{in} \quad \Omega_1, \\
{\frac{3}{19}}, \qquad \text{in} \quad \Omega_2.
\end{cases}
\end{align}
In this case, the exact solution is a standing wave
\begin{align}\label{example3_3_3d}
u= 
\begin{cases}
\sin(\sqrt{24}\pi t)\cos(4\pi x_1) \cos(2\pi x_2 ) \cos(2\pi x_3), \qquad \text{in} \quad \Omega_1, \\
\sin(\sqrt{24}\pi t)\cos(12\pi x_1 )\cos(2\pi x_2) \cos(2\pi x_3), \qquad \text{in} \quad \Omega_2.
\end{cases}
\end{align}

This numerical example has been presented in \cite{huang2020adaptive}. Here, we show the sample code and focus more on the CPU cost study.

 The main code is included in the GitHub repository \verb|./example/03_wave_02_heter_media.cpp|. Since the refinement and coarsening procedures are the same as the Hamilton-Jacobi equation, we only present the time evolution part. Here, we use the interior penalty DG \cite{arnold1982interior} for the spatial discretization. In the weak formulation, we assemble the matrix for the linear part $-\frac{\sigma}{h}\sum_{i,j} [u]_{i+\half,j+\half} [v]_{i+\half,j+\half}$ where $\sigma$ is the penalty constant, $h$ is the mesh size, $u$ is the solution and $v$ is the test function. Then, we use Lagrangian interpolation to approximate $c^2 u_x$, $c^2 u_y$, $(c^2)^- u$ and $(c^2)^+ u$ and apply fast matrix-vector multiplication to evaluate the residual terms.
\begin{lstlisting}
// --- part 4: time evolution
// -sigma/h * [u] * [v]
DiffusionAlpt operator_ujp_vjp(dg_solu, oper_matx_alpt, sigma_ipdg);
operator_ujp_vjp.assemble_matrix_flx_ujp_vjp();	

dg_solu.set_source_zero();

RK4ODE2nd odesolver(operator_ujp_vjp, dt);
odesolver.init();

for (size_t stage = 0; stage < odesolver.num_stage; stage++)
{		
	dg_solu.set_rhs_zero();

	// interpolation of k * u_x and k * u_y
	interp.var_coeff_gradu_Lagr_fast(coe_func, is_intp, fastLagr);
	
	diffuseRHS.rhs_flx_gradu();
	diffuseRHS.rhs_vol();

	// interpolation of k- * u
	interp.var_coeff_u_Lagr_fast(coe_func_minus, is_intp_d0, fastLagr);
	diffuseRHS.rhs_flx_k_minus_u();

	// interpolation of k+ * u
	interp.var_coeff_u_Lagr_fast(coe_func_plus, is_intp_d0, fastLagr);
	diffuseRHS.rhs_flx_k_plus_u();
	
	odesolver.rhs_to_eigenvec();
	
	// [u] * [v]
	odesolver.add_rhs_matrix(operator_ujp_vjp);			
	
	odesolver.step_stage(stage);

	odesolver.final();
}	
\end{lstlisting}

In Table \ref{tb:dis-coeff-adaptive}, we show the convergence study for $d = 2, 3$ and $k = 1, 2, 3$. For $d=2,3$ with $k=1$, the errors saturate because we use   $\epsilon=10^{-5}$. Moreover, to reach the same magnitude of error, the DOF with higher polynomial degree will be much less that that with lower polynomial degree. In Figure \ref{fig:wave-cpu-err}, we plot the $L^2$ errors versus the average CPU cost per time step. It is again observed that, the errors with larger polynomial degree $k$ decays faster when reducing the refinement threshold $\epsilon$ and thus increasing DOF. The slopes with the same polynomial degrees in Figure \ref{fig:wave-cpu-err} are almost the same with different dimensions. This indicates that the proposed method is resistant to the curse of dimensionality.

\begin{table}[!hbp]
\centering
\caption{wave equation with $d=2,3$ and $k=1,2,3$. Adaptive sparse grid DG. $N=8$, $t=0.01$.}
\label{tb:dis-coeff-adaptive}
\begin{tabular}{c|c|c|c|c|c|c|c|c|c}
\hline
\multicolumn{2}{c|}{} & \multicolumn{4}{c|}{$d=2$} & \multicolumn{4}{c}{$d=3$} \\
\hline
& $\epsilon$ & DOF & $L^2$-error  & $R_{\epsilon}$ & $R_{\textrm{DOF}}$& DOF & $L^2$-error  & $R_{\epsilon}$ & $R_{\textrm{DOF}}$ \\
\hline
\multirow{5}{3em}{$k=1$}
& 1e-1	&	784		&	1.97e-3	&	-		&	-		&	3840	&	4.96e-3	&	-	&	-	\\
& 1e-2	&	2816	&	4.17e-4	&	0.68	&	1.22	&	16832	&	1.24e-3	&	0.60	&	0.94	\\
& 1e-3	&	8496	&	7.31e-5	&	0.76	&	1.58	&	70656	&	1.37e-4	&	0.96	&	1.54	\\
& 1e-4	&	24384	&	4.18e-5	&	0.24	&	0.53	&	249664	&	3.86e-5	&	0.55	&	1.00	\\
& 1e-5	&	56960	&	4.10e-5	&	0.01	&	0.02	&	862528	&	3.19e-5	&	0.08	&	0.15	\\
\hline
\multirow{5}{3em}{$k=2$}
& 1e-1	&	3204	&	7.72e-3	&	-		&	-		&	5238	&	1.12e-3	&	-	&	-	\\
& 1e-2	&	13842	&	1.17e-3	&	0.82	&	1.29	&	16794	&	1.71e-4	&	0.82	&	1.61	\\
& 1e-3	&	23544	&	1.12e-4	&	1.02	&	4.42	&	43416	&	3.33e-5	&	0.71	&	1.72	\\
& 1e-4	&	40320	&	1.39e-5	&	0.91	&	3.88	&	117558	&	5.32e-6	&	0.80	&	1.84	\\
& 1e-5	&	75042	&	6.32e-6	&	0.34	&	1.26	&	691632	&	1.10e-6	&	0.69	&	1.73	\\
\hline
\multirow{5}{3em}{$k=3$}
& 1e-1	&	1472	&	3.29e-3	&	-	    &	-	 	& 2304 		& 5.59e-4 	& - & -   \\ 
& 1e-2 	&	3904	&	8.00e-5	&	0.61	&	1.45 	& 7040 		& 1.28e-4 	& 0.64 & 1.32 	\\ 
& 1e-3 	&	8256	&	7.71e-5	&	1.02	&	3.13 	& 18176 	& 1.65e-5 	& 0.89 & 2.16 	\\
& 1e-4 	&	23488	&	1.47e-5	&	0.72	&	1.58 	& 41472 	& 1.55e-6 	& 1.03 & 2.87 	\\
& 1e-5 	&   39360	&	1.15e-6 &	1.11	&	4.94	& 92928		& 2.58e-7	& 0.78 & 2.22 	\\
\hline
\end{tabular}
\end{table}

\begin{figure}
    \centering
    \includegraphics[width=0.7\textwidth]{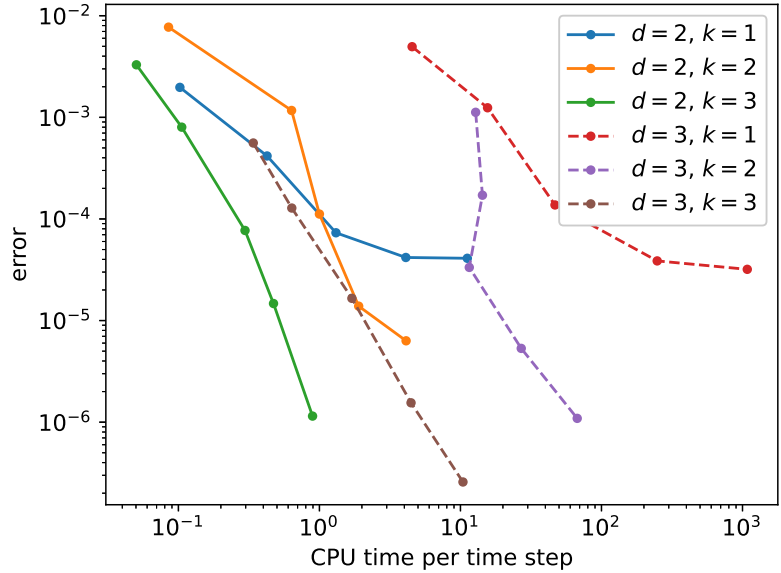}
    \caption{wave equation with $d=2,3$ and $k=1,2,3$. $L^2$-error versus CPU time per time step. Adaptive sparse grid DG. $N=8$, $t=0.01$.}
    \label{fig:wave-cpu-err}
\end{figure}

%% file: conclusion.tex
\section{Conclusions and  future work}
\label{sec:conclusion}

In this paper, we showcased the main components of our adaptive sparse grid DG C\texttt{++} package AdaM-DG for solving PDEs. We focused on the details of the implementation, including the data structure, assembling of the operators, fast algorithms, and further demonstrated how to implement the package by three examples. 

We would like to emphasize that this is still an on-going project, and many efforts are still needed to improve the software in various aspects. In particular, in the near term, we would like to improve some details of the code, including a more flexible definition of the computational domain, C\texttt{++} implementation to further improve the memory and CPU time efficiency, a  clear and comprehensive code documentation, a user-friendly Python interface, and a more efficient linear/nonlinear solver such as PETSc \cite{balay2019petsc}. More importantly, we would like to generalize the code to an HPC platform with efficient parallel implementations on multicore CPU (using MPI) and GPU (using CUDA). Many aspects of the computational algorithms can also be further developed, including the hybridization with other numerical discretizations and a multi-domain approach which is more friendly to heterogeneous computing architecture.

\section*{Acknowledgements}
We acknowledge the High Performance Computing Center (HPCC) at Michigan State University for providing computational resources that have contributed to the research results reported within this paper.
We also would like to thank Kai Huang from Michigan State University, Yuan Liu from Wichita State University, Zhanjing Tao from Jilin University, and Qi Tang from Los Alamos National Laboratory for the assistance and discussion in the code implementation.